\documentclass[12pt]{amsart}
\usepackage{latexsym}
\usepackage{amssymb}
\usepackage{amsmath}

\usepackage{color}

\newtheorem{Thm}[subsection]{Theorem}
\newtheorem{Lemma}[subsection]{Lemma}
\newtheorem{Rem}[subsection]{Remark}
\newtheorem{Def}[subsection]{Definition}
\newtheorem{Cor}[subsection]{Corollary}
\newtheorem{Prop}[subsection]{Proposition}
\numberwithin{equation}{section}

\newcommand{\ds}{\displaystyle}

\newcommand{\ben}{\begin{enumerate}}
\newcommand{\een}{\end{enumerate}}

\newcommand{\bec}{\begin{center}}
\newcommand{\eec}{\end{center}}

\newcommand{\beq}{\begin{equation}}
\newcommand{\eeq}{\end{equation}}

\newcommand{\bdm}{\begin{displaymath}}
\newcommand{\edm}{\end{displaymath}}

\newcommand{\R}{\mathbb{R}}

\newcommand{\C}{\mathbb{C}}

\newcommand{\Z}{\mathbb{Z}}
\newcommand{\Q}{\mathbb{Q}}
\newcommand{\cS}{\mathcal{S}}
\newcommand{\N}{\mathbb{N}}

\newcommand{\cA}{\mathcal{A}}
\newcommand{\cP}{\mathcal{P}}
\newcommand{\cD}{\mathcal{D}}

\newcommand{\ffi}{\varphi}

\newcommand{\de}{\delta}
\newcommand{\al}{\alpha}

\newcommand{\ga}{\gamma}
\newcommand{\supp}{\mathrm{supp}\,}
\newcommand{\esssup}{\mathop{\rm ess\,sup}}

\newenvironment{proofof}[1]{{\sc Proof of #1}}{\quad\lower0.05cm\hbox{$\square$}\medskip}

\title[Littlewood-Paley Theory with Matrix Weights]
{Littlewood-Paley Theory for\\ Matrix-Weighted Function Spaces}

\author{Michael Frazier}
\address{Mathematics Department, University of Tennessee, Knoxville, Tennessee 37922} 
\email{mfrazie3@utk.edu}

\author{Svetlana Roudenko}
\address{Department of Mathematics \& Statistics, Florida International University, Miami, FL 33199}
\email{sroudenko@fiu.edu}



\date{} 

\begin{document}

\begin{abstract}
We define the vector-valued, matrix-weighted function spaces $\dot{F}^{\alpha q}_p(W)$ (homogeneous) and $F^{\alpha q}_p(W)$ (inhomogeneous) on $\R^n$, for $\alpha \in \R$, $0<p<\infty$, $0<q \leq \infty$, with the matrix weight $W$ belonging to the 
$A_p$ class.  For $1<p<\infty$, we show that $L^p(W) = \dot{F}^{0 2}_p(W)$, and, for $k \in \N$, that $F^{k 2}_p(W)$ coincides with the matrix-weighted Sobolev space $L^p_k(W)$, thereby obtaining Littlewood-Paley characterizations of $L^p(W)$ and $L^p_k (W)$.  We show that a vector-valued function belongs to $\dot{F}^{\alpha q}_p(W)$ if and only if its wavelet or $\ffi$-transform coefficients belong to an associated sequence space $\dot{f}^{\alpha q}_p(W)$. We also characterize these spaces in terms of reducing operators associated to $W$.
\end{abstract}

\maketitle

\tableofcontents

\section{Introduction}

Littlewood-Paley theory originated with the development of certain auxiliary integral expressions used in the study of analytic functions and Fourier series (see e.g., \cite{S3}, \cite{S1}, and \cite{FJW} for background).  This theory was extended to $\R^n$ by Stein and others (\cite{S0}, \cite{BCP}) and these auxiliary expressions were found to be useful in studying function spaces.  In the 1970s a systematic approach to function spaces using variants of the classical Littlewood-Paley expressions was developed by Peetre, Triebel, and others (see e.g., \cite{T} for more information).  In particular, most standard function spaces other than $L^1$ or $L^{\infty}$ fit into two scales of spaces, the Besov and Triebel-Lizorkin spaces, which are defined via expressions of Littlewood-Paley type.  This theory meshed perfectly with wavelet theory to provide characterizations of the function spaces in these two scales in terms of the magnitudes of wavelet coefficients (see e.g., \cite{M2} or \cite{FJW}).  

The theory of (scalar) $A_p$ weights originated in Muckenhoupt \cite{Mu} and Hunt, Muckenhoupt, and Wheeden \cite{HMW}.  Much of the Littlewood-Paley theory extends to the case of (scalar) weighted function spaces (see \cite[\S 10]{FJ2} ).  Matrix weights were developed in the 1990s, starting with \cite{TV} and \cite{NT}.  Matrix-weighted Besov spaces were defined and developed in \cite{R}, \cite{R1}, \cite{R2}, and \cite{FR}.  For recent developments on matrix weights see \cite{CUIMPRR}, \cite{CUIM}; for an application of matrix weights to elliptic systems see \cite{IM}.  

Our goal is to adapt Littlewood-Paley theory to matrix-weighted Triebel-Lizorkin spaces, which we will see includes the matrix-weighted $L^p$ and Sobolev spaces, when the weight belongs to the 
matrix $A_p$ class. In particular, we obtain characterizations of these spaces in terms of the magnitudes of wavelet coefficients.

To state results, we first need some notation.  The side length of any cube $Q \subseteq \R^n$ is denoted by $\ell(Q)$.  For $j \in \Z$ and $k = (k_1, k_2, \dots, k_n) \in \Z^n$, let $Q_{j,k}= \Pi_{i=1}^n [2^{-j}k_i, 2^{-j} (k_i +1)]$ be the dyadic cube of side length $\ell(Q_{j,k})= 2^{-j}$ and ``lower left corner" $x_Q= 2^{-j}k$.  Let $\mathcal{D} = \{ Q_{j,k} \}_{j \in \Z, k \in \Z^n}$ denote the collection of all dyadic cubes in $\R^n$, and let $\mathcal{D}_j = \{Q \in \mathcal{D}: \ell(Q)=2^{-j}\}$.  

Let $\cS$ denote Schwartz space, let $\cS'$ be its dual, and let $\mathcal{P}$ be the class of the polynomials, all on $\R^n$.  We fix a positive integer $m$ and consider vector-valued functions $\vec{f} = (f_1,...,f_m)^T$ on $\R^n$.  Generally we require that each component $f_i$ belongs to $\cS'/{\mathcal{P}}$, the space of tempered distributions modulo polynomials; in that case we write $\vec{f} \in \cS'/{\mathcal{P}}$. We will consider sequences $\vec{s} = \{ \vec{s}_Q \}_{Q \in \mathcal{D} }$, where for each $Q \in \mathcal{D}$, $\vec{s}_Q = ( (s_Q)_1, (s_Q)_2, \dots , (s_Q)_m )^T \in \C^m$.

We say that a function $\ffi:\R^n \rightarrow \C$ is admissible, and we write $\ffi \in \mathcal{A}$, if 
\begin{equation} \label{admiss1}
\varphi \in \mathcal{S} (\R^n),
\end{equation}
\begin{equation} \label{admiss2}
\mbox{supp} \,\,\, \hat{\varphi} \subseteq \{ \xi: 1/2 \leq |\xi| \leq 2\}
\end{equation} 
and 
\begin{equation} \label{admiss3}
|\hat{\phi}(\xi)| \geq c >0 \quad \mbox{if} \,\,\,  3/5 \leq |\xi| \leq 5/3.
\end{equation}


For $j \in \Z$, let $\varphi_j (x) = 2^{jn} \varphi (2^j x)$.  We define convolution of the scalar function $\ffi_j$ with $\vec{f}$ componentwise: $\ffi_j \ast \vec{f} = (\ffi_j \ast f_1,..., \ffi_j \ast f_m)^T$.  A matrix weight $W$ is a map on $\R^n$ such that $W(x)$ is a non-negative definite $m\times m$ matrix for each $x \in \R^n$, where $W$ is a.e. invertible and the entries of $W$ are measurable functions on $\R^n$.

For definitions (i)-(iv) below, we suppose $\alpha \in \R$, $0 < p < \infty$, $0< q \leq \infty, \ffi \in \mathcal{A}$, and $W $ is a matrix weight.

(i) The Triebel-Lizorkin space $\ds \dot{F}^{\alpha q}_p(W)$ is the set
of all $\vec{f} \in \cS'/{\mathcal P}(\R^n)$ such that
$$
\Vert \vec{f} \, \Vert_{\dot{F}^{\alpha q}_p(W)} =
\left\Vert \left(  \sum_{j \in \Z}  \left|  2^{j \al}  W^{1/p} \ffi_j \ast \vec{f} \, \right|^q  \right)^{1/q} \right\Vert_{L^p(\R^n)}  < \infty .
$$

(ii) The discrete Triebel-Lizorkin space $\ds \dot{f}^{\alpha q}_p(W)$ is the set
of all sequences $\vec{s} = \{ \vec{s}_Q \}_{Q \in \mathcal{D} }$ such that
$$
\Vert \vec{s} \, \Vert_{\dot{f}^{\alpha q}_p(W)} =
\left\Vert \left(  \sum_{Q \in \mathcal{D} } \left(   |Q|^{-\al /n -1/2} \vert \, W^{1/p} \vec{s}_Q \, \vert \chi_Q \right)^q  \right)^{1/q} \right\Vert_{L^p(\R^n)}  < \infty .
$$

Suppose that for each $Q \in \mathcal{D}$, $A_Q$ is an $m \times m$ non-negative definite matrix. 

(iii)  The $\{ A_Q \}$- Triebel-Lizorkin space $\ds \dot{F}^{\alpha q}_p(\{A_Q \} )$ is the set
of all $\vec{f} \in \cS'/{\mathcal P}(\R^n)$ such that
$$
\Vert \vec{f} \, \Vert_{\dot{F}^{\alpha q}_p(\{A_Q \})} = \left\Vert \left(  \sum_{j \in \Z}   \sum_{Q \in \mathcal{D}_j} \left( 2^{j \al} \vert \, A_Q \, \ffi_j \ast \vec{f} \, \vert \chi_Q \right)^q  \right)^{1/q} \right\Vert_{L^p(\R^n)}   < \infty .
$$

(iv)  The $\{ A_Q \}$-discrete Triebel-Lizorkin space $\ds \dot{f}^{\alpha q}_p(\{A_Q\})$ is the set
of all sequences $\vec{s} = \{ \vec{s}_Q \}_{Q \in \mathcal{D} }$ such that
$$
\Vert \vec{s} \, \Vert_{\dot{f}^{\alpha q}_p(\{A_Q \})} =
\left\Vert \left(  \sum_{Q \in \mathcal{D} } \left(   |Q|^{-\al /n -1/2} \vert \, A_Q \vec{s}_Q \, \vert \chi_Q \right)^q  \right)^{1/q} \right\Vert_{L^p(\R^n)}  < \infty .
$$

In all cases, when $q=\infty$, the $\ell^q$ quasi-norm is replaced with the supremum.  Note that if we set $t_Q = \vert \, A_Q \vec{s}_Q \, \vert$ and $t = \{ t_Q \}_{Q \in \mathcal{Q} }$, then 
\begin{equation} \label{reduction}
 \Vert \vec{s} \, \Vert_{\dot{f}^{\alpha q}_p(\{A_Q \})} = \Vert t \, \Vert_{\dot{f}^{\alpha q}_p},
\end{equation}
where $\dot{f}^{\alpha q}_p$ is the usual scalar, unweighted discrete Triebel-Lizorkin space.  This fact will sometimes allow us to deduce results for the matrix-weighted spaces from the corresponding scalar, unweighted results, such as in Theorem \ref{ADmatrixbnded} below.  

Our goal is to prove equivalences of these spaces, when $\vec{s} = \{ \vec{s}_Q \}_{Q \in \mathcal{Q}}$ is the sequence of $\ffi$-transform coefficients of $\vec{f}$ (and similarly, for wavelet coefficients), and $\{ A_Q \}_{Q \in \mathcal{Q}}$ is a sequence of reducing operators of order $p$ for a matrix weight $W \in A_p$, defined as follows.

Given any matrix weight $W$ and $0< p < \infty$, there exists (see e.g., \cite[Proposition 1.2]{G} for $p>1$ and \cite[p. 1237]{FR} for $0<p\leq 1$) a sequence $\{A_Q\}_{Q \in \mathcal{D}}$ of positive definite $m \times m$ matrices such that 
\[  c_1 | A_Q \vec{y} |\leq \left( \frac{1}{|Q|} \int_Q \Vert W^{1/p} (x) \vec{y} \Vert^p \, dx \right)^{1/p} \leq c_2 | A_Q \vec{y} |, \]
with positive constants $c_1, c_2$ independent of $\vec{y} \in \C^m$ and $Q \in \mathcal{D}$.  In this case, we call $\{A_Q\}$ a sequence of reducing operators of order $p$ for $W$.

The matrix $A_2$ class was first defined in \cite{TV}, and $A_p$, for other $p \in (1, \infty)$, in \cite{NT}.  We use the following characterization, proved in \cite{R}:
$W \in A_p (\R^n)$ ($1<p<\infty$) if and only if
\[ \sup_Q \frac{1}{|Q|} \int_Q \left( \frac{1}{|Q|} \int_Q  \Vert  W^{1/p}(x) W^{-1/p} (y)      \Vert^{p^{\, \prime}}   \, dy   \right)^{p/p^{\, \prime}} \, dx < \infty,  \]
where $\Vert \cdot \Vert$ is the operator norm of the matrix, $p^{\, \prime} = p/(p-1)$ is the conjugate index of $p$, and the supremum is taken over all cubes $Q \subset \R^n$.  For $0<p\leq 1$, we use the definition from \cite{FR}:  $W \in A_p$ if 
\begin{equation} \label{defAp2}
  \sup_Q \,\, \esssup_{y \in Q} \frac{1}{|Q|} \int_Q \Vert W^{1/p} (x) W^{-1/p} (y) \Vert^p \, dx < \infty. 
\end{equation}  

Since $\varphi_j (x) = 2^{jn} \varphi (2^j x)$, we have $\widehat{\varphi_j} (\xi)= \hat{\varphi} (2^{-j} \xi)$.  For $\ffi \in \mathcal{A}$, let $\hat{\psi} = \frac{\hat{\varphi}}{\sum_{j \in \Z} |\widehat{\varphi_j}|^2}$.  Then $\psi \in \mathcal{A}$, and we have $\sum_{j \in \Z} \overline{\widehat{\varphi_j}}(\xi) \widehat{\psi_j} (\xi)  =1$ for all $\xi \neq 0$.  For $Q = Q_{j,k}$, we define 
\begin{equation} \label{defphiQ}
  \varphi_Q (x) = 2^{jn/2} \varphi (2^jx-k) = |Q|^{-1/2} \varphi ((x-x_Q)/\ell(Q)),   
\end{equation}
and similarly for $\psi_Q$.  Recall that $\cS'/{\mathcal{P}}$ is the dual of $\cS_0 = \{ g \in \mathcal{S}: D^{\alpha} \hat{g}(0) =0$ 
{for all multi-indices} $\alpha\}$, see e.g., \cite[p. 237]{T}.  We use the notation $\langle f, g \rangle$ to denote a pairing which is linear in $f$ and conjugate linear in $g$; when this pairing is between a distribution $f$ and a test function $g$, then $\langle f, g \rangle = f(\overline{g})$.  Then we have the ``$\varphi$-transform" identity $f = \sum_{Q \in \mathcal{D}} \langle f, \varphi_Q \rangle \psi_Q $, with convergence in $L^2$ if $f \in L^2$, convergence in $\cS$ if $f \in \cS_0$, and convergence in $\cS'/{\mathcal{P}}$ if $f \in \cS'$ (see \cite{FJ}, \cite{FJ1}, or \cite {BT}, Theorem 2.4 for details about the $\varphi$-transform).  For vector-valued functions $\vec{f}$, we define $\langle \vec{f} , g\rangle = (\langle f_1, g \rangle, \cdots, \langle f_m, g\rangle)^T$.  Then we have 
\begin{equation} \label{phitraniden}
\vec{f} = \sum_{Q \in \mathcal{D}} \langle \vec{f}, \varphi_Q \rangle \psi_Q ,
\end{equation}
with convergence as noted above, in each component.

The notation $\Vert z \Vert_X \approx \Vert z \Vert_Y$, for quasi-normed spaces $X$ and $Y$, will always mean that the quasi-norms are equivalent: $X=Y$ as sets, and there exist positive constants $c_1, c_2$ independent of $z$ such that $c_1 \Vert z \Vert_X \leq  \Vert z \Vert_Y \leq c_2 \Vert z \Vert_X$ for all $z$.   

We now state the results of this paper. The main statement is the following theorem, connecting matrix-weighted Triebel-Lizorkin spaces with their discrete or sequence space analogs. 

\begin{Thm}\label{mainresult} Suppose $\alpha \in \R, 0 < p < \infty, 0 < q \leq \infty, \ffi \in \mathcal{A}, W \in A_p(\R^n)$, and $\{ A_Q \}_{Q \in \mathcal{D}}$ is a sequence of reducing operators of order $p$ for $W$.  For $\vec{f} \in \cS^{\, \prime}/\cP$, let $\vec{s} = \{ \vec{s}_Q \}_{Q \in \mathcal{D}}$, where $\vec{s}_Q = \langle \vec{f}, \ffi_{Q} \rangle$.  Then if any of $\Vert \vec{f} \Vert_{\dot{F}^{\alpha q}_p (W)}$, $\Vert \vec{f} \Vert_{\dot{F}^{\alpha q}_p (\{A_Q \})}$, $\Vert  \vec{s} \Vert_{\dot{f}^{\alpha q}_p ( W )}$, or $\Vert\vec{s} \Vert_{\dot{f}^{\alpha q}_p ( \{ A_Q \} )}$ is finite, then so are the other three, with  
\[ \Vert \vec{f} \Vert_{\dot{F}^{\alpha q}_p (W)} \approx \Vert \vec{f} \Vert_{\dot{F}^{\alpha q}_p (\{A_Q \})} \approx \Vert\vec{s} \Vert_{\dot{f}^{\alpha q}_p ( \{ A_Q \} )} \approx \Vert  \vec{s} \Vert_{\dot{f}^{\alpha q}_p ( W )}. \]
Also, $\dot{F}^{\alpha q}_p (W)$ and $\dot{F}^{\alpha q}_p (\{A_Q \})$ are independent of the choice of $\varphi \in \mathcal{A}$, in the sense that different choices yield equivalent quasi-norms.
\end{Thm}

The next statement is an adaptation of Theorem \ref{mainresult} to expansions based on wavelets instead of the $\varphi$-transform.  We start by recalling wavelets.  A wavelet basis is an orthonormal basis for $L^2 (\R^n)$ of the form $\{ \psi^{(i)}_Q \}_{Q \in \mathcal{D}, 1 \leq i \leq 2^n -1} $, where $\{ \psi^{(i)} \}_{i=1}^{2^n-1}$ are the generators of the wavelet basis, and $\psi^{(i)}_Q(x) = |Q|^{-1/2} \psi^{(i)} ((x - x_Q)/\ell(Q))$, similarly to (\ref{defphiQ}).   For $W \in A_p$, we obtain a characterization of $\dot{F}^{\alpha q}_p ( W )$ in terms of the wavelet coefficients, for wavelets with appropriate properties.

\begin{Thm}\label{wavechar}
Suppose $\alpha \in \R, 0 < p < \infty, 0 < q \leq \infty$, and $W \in A_p(\R^n)$.  Suppose that for some sufficiently large positive numbers $N_0, R$, and $S$ (depending on $p, q, \alpha, n$, and $W$), the generators $\{ \psi^{(i)} \}_{1 \leq i \leq 2^n -1} $ of a wavelet basis satisfy $\int_{\R^n} x^{\gamma} \psi^{(i)} (x) \, dx =0$ for all multi-indices $\gamma$ with $|\gamma| \leq N_0$, and $|D^{\gamma} \psi^{(i)} (x)| \leq C (1+|x|)^{-R}$ for all $|\gamma| \leq S$.  Then 
\[ \Vert \vec{f} \Vert_{\dot{F}^{\alpha q}_p ( W )} \approx \sum_{i=1}^{2^n -1} \Vert \{ \langle \vec{f}, \psi^{(i)}_Q \rangle \}_{Q \in \mathcal{D}} \Vert_{\dot{f}^{\alpha q}_p ( W )}  . \]
\end{Thm}

Examples of wavelets with the properties in Theorem \ref{wavechar} are Meyer's wavelets (see \cite{M} and \cite{LM}) and Daubechies' $D_N$ wavelets for sufficiently large $N$ (\cite{D}).  

As in the unweighted case, the spaces $L^p(W)$ (defined as the set of measurable $\vec{f}$ such that $\Vert \vec{f} \Vert_{L^p(W)}^p = \int_{\R^n} |W^{1/p} (x) \vec{f}(x)|^p \, dx < \infty$), are contained in the scale of Triebel-Lizorkin spaces.

\begin{Thm}\label{equivwithlp}
Suppose $1 < p < \infty$ and $W \in A_p(\R^n)$.  Then $\dot{F}^{0 2}_p (W) = L^p(W)$, with equivalent norms.
\end{Thm}

The interpretation of the equality $\dot{F}^{0 2}_p (W) = L^p(W)$ is that if $\vec{f} \in L^p(W)$, then the equivalence class of $\vec{f}$ in $\cS^{\, \prime}/\cP$ belongs to $\dot{F}^{0 2}_p (W)$, and any equivalence class in $\dot{F}^{0 2}_p (W)$ has a unique representative belonging to $L^p(W)$.

In \cite[Theorem 15.1]{NT}, Nazarov and Treil (see also \cite{V}) prove that for $n=1, W \in A_p$, and a sufficiently nice wavelet system (as in Theorem \ref{wavechar}),  
\[ \Vert \vec{f} \Vert_{L^p(W)} \approx \sum_{i=1}^{2^n -1} \Vert \{ \langle \vec{f}, \psi_{Q,i} \rangle \}_{Q \in \mathcal{Q}} \Vert_{\dot{f}^{0 2}_p ( \{ A_Q \} )}. \]
Assuming this result, Theorem \ref{equivwithlp} for $n=1$ follows from Theorem \ref{wavechar} and Theorem \ref{nonavnormeq} below. 

As in the classical case, there are inhomogeneous analogues, denoted $F^{\alpha q}_p (W)$, of the homogeneous spaces $\dot{F}^{\alpha q}_p (W)$.  For the inhomogeneous spaces, the terms involving $\varphi_j \ast \vec{f}$ for $j <1$ are replaced by a single term $\Phi \ast \vec{f}$.  The corresponding sequence elements $\vec{s}_Q$ are indexed by cubes $Q$ with $\ell(Q) \leq 1$ only.  These inhomogeneous spaces are spaces of tempered distributions rather than tempered distributions modulo polynomials.  The theory for the inhomogeneous spaces is entirely analogous to the theory in the homogeneous case.  In particular, we will see that $F^{0 2}_p (W)\approx L^p(W)$ for $1<p< \infty$.  One advantage of the inhomogeneous spaces is that they include the Sobolev spaces for $1<p<\infty$, defined in the matrix-weighted case as follows.

For $\beta = (\beta_1, \dots, \beta_n)$ a multi-index (so $\beta_i \in \Z$ with $\beta_i \geq 0$ for all $i$), let $|\beta| = \sum_{i=1}^n \beta_i$ and let $D^{\beta}= \partial_1^{\beta_1} \cdots \partial_n^{\beta_n}$.  For $\vec{f} \in \cS^{\, \prime} (\R^n)$, let $D^{\beta}\vec{f} = (D^{\beta}f_1, \dots, D^{\beta}f_m)$.  For $k \in \N, 1<p<\infty$, and $W$ a matrix weight, define the matrix-weighted Sobolev space $L^p_k (W)$ to be the set of all $\vec{f} \in \cS^{\, \prime} (\R^n)$ such that 
\[  \Vert \vec{f} \Vert_{L^p_k (W)} \equiv \sum_{\beta: |\beta|\leq k} \Vert D^{\beta} \vec{f} \Vert_{L^p(W)} < \infty.  \]

\begin{Prop} \label{Sobolevequiv}
Suppose $k \in \N, 1<p<\infty$, and $W \in A_p (\R^n)$.  Then $L^p_k(W) = F^{k\, 2}_p (W)$, with equivalent norms.
\end{Prop}

The paper is organized as follows: we prove the equivalence between the averaging spaces $\dot{F}^{\alpha q}_p(\{A_Q \} )$ and $\dot{f}^{\alpha q}_p(\{A_Q \} )$ in Theorem \ref{avgnormeq}; this is be done by variations on the methods used for the scalar theory and is discussed in Section 2. The equivalences between the weighted spaces and averaged spaces, in both the function case and the sequence case, are stated and proved in Theorem \ref{nonavnormeq}; the proofs involve some less familiar techniques, which are discussed in Section 3. Theorem \ref{mainresult} follows from Theorems \ref{avgnormeq} and \ref{nonavnormeq}. Theorem \ref{wavechar} follows from Theorem \ref{mainresult} and Theorem \ref{avgwavechar}.  Theorem \ref{equivwithlp} is proved in Section \ref{Lpequiv}. We define and discuss the inhomogeneous spaces $F^{\alpha q}_p (W)$ in Section \ref{inhomog}. Finally, the equivalence with Sobolev spaces (Proposition \ref{Sobolevequiv}) is proved in Section \ref{Sobolev}.

\vspace{0.1in}

\noindent{{\bf Acknowledgments.}}  We thank Fedor Nazarov, who provided us with the formulation and proof of Theorem \ref{Naz}. S.R. was partially supported by the NSF-DMS CAREER grant \# 1151618/1929029.

\section{Equivalence of the averaging spaces}\label{normeqavg}

We show 
that the equivalence of the averaging spaces $\dot{F}^{\alpha q}_p(\{A_Q \} )$ and $\dot{f}^{\alpha q}_p(\{A_Q\})$ holds under just the strong doubling assumption on $\{A_Q\}$, defined as follows (see \cite[Definition 1.3]{R2}).

\begin{Def} \label{AQdoubling}  Let $\{ A_Q \}_{Q \in \mathcal{D}}$ be a sequence of nonnegative-definite matrices and let $\beta, p >0$.  We say that $\{ A_Q \}$ is strongly doubling of order $(\beta, p)$ if there exists $c>0$ such that 
\beq  
\Vert A_Q A_P^{-1} \Vert^p \leq c \max \left\{ \left( \frac{\ell(P)}{\ell(Q)} \right)^n , \left( \frac{\ell(Q)}{\ell(P)} \right)^{\beta -n}  \right\}  \left( 1 + \frac{|x_Q - x_P|}{\max \{ \ell(P), \ell(Q) \}} \right)^{\beta},  \label{stdoubest}
\eeq
for all $Q,P \in \mathcal{D}$.  We say $\{ A_Q \}$ is weakly doubling of order $r>0$ if there exists $c >0$ such that 
\beq 
 \Vert A_{Q_{j k}} A_{Q_{j \ell}}^{-1} \Vert \leq c ( 1 + |k - \ell|)^{r},  \label{weakdoubest}
\eeq 
for all $k, \ell \in \Z^n$ and all $j \in \Z$. 
\end{Def}

A strongly doubling sequence of order $(\beta, p)$ is weakly doubling of order $r = \beta/p$, because (\ref{weakdoubest}) is just the restriction of (\ref{stdoubest}) to the case when $\ell(P)= \ell(Q)$.  

A  matrix weight $W$ is called a {\em doubling matrix weight of order} $p>0$ if the scalar measures $w_{\vec{y}} (x) = |W^{1/p} (x) \vec{y}|^p$, for $\vec{y} \in \C^m$, are uniformly doubling: there exists $c>0$ such that for all cubes $Q \subseteq \R^n$ and all $\vec{y} \in \C^m$,
$\int_{2Q} w_{\vec{y}} (x) \, dx \leq c \int_{Q} w_{\vec{y}} (x) \, dx$, where $2Q$ is the cube concentric with $Q$, having twice the side length of $Q$.  If $c=2^{\beta}$ is the smallest constant for which this inequality holds, we say that $\beta$ is the doubling exponent of $W$.  If $W \in A_p$, then $W$ is a doubling matrix weight (for $0<p\leq 1$, see \cite {FR}, Lemma 2.1; for $p>1$, this fact follows because the scalar weights $w_{\vec{y}}$ are uniformly in the scalar $A_p$ class (\cite{V}, Lemma 5.3), and hence, are uniformly doubling, \cite[p. 196]{S2}).

The following lemma explains the connection between doubling weights $W$ and doubling sequences $\{ A_Q \}$.

\begin{Lemma}\label{doublemma}
Let $W$ be a doubling matrix weight of order $p>0$ with doubling exponent $\beta$ and suppose $\{ A_Q \}_{Q \in \mathcal{D}}$ is a sequence of reducing operators of order $p$ for $W$.  Then $\{ A_Q \}$ is strongly doubling of order $(\beta, p)$.
\end{Lemma}

\begin{proof}
For $\vec{y} \in \C^m$, let $w_{\vec{y}} (x) = \vert W^{1/p}(x) \vec{y} \vert^p$.  Fix $P, Q \in \mathcal{D}$ and let $j $ be the smallest nonnegative integer such that $Q \subseteq 2^j P$.  Then 
\beq
 2^j \leq c \max \left\{1, \frac{\ell(Q)}{\ell(P)}    \right\} \left(1 + \frac{|x_P - x_Q|}{\max \{\ell(P), \ell(Q) \} }  \right) . \label{2jest}
\eeq
By the doubling property,
\[ w_{\vec{y}} (Q) \leq w_{\vec{y}} (2^j P)  \leq 2^{\beta j} w_{\vec{y}} ( P) . \]
Therefore, 
\[ \vert A_Q \vec{y} \vert^p \leq c  \frac{1}{|Q|} \int_Q  \vert W^{1/p}(x) \vec{y} \vert^p \, dx  = c  \frac{1}{|Q|}  w_{\vec{y}} (Q)  \]
\[ \leq c  \frac{1}{|Q|} 2^{\beta j} w_{\vec{y}} ( P)  = c \frac{|P|}{|Q|} 2^{\beta j} \vert A_P \vec{y} \vert^p . \]
Substituting $\vec{y} = A_P^{-1}\vec{z}$ for arbitrary $\vec{z}$ and applying \eqref{2jest} yields the conclusion.
\end{proof}

\begin{Thm}\label{avgnormeq}
Let $0 < p < \infty$, $0 < q \leq \infty$, $\alpha \in \R$ and $\ffi \in \mathcal{A}$.  Suppose $\{ A_Q \}_{Q \in \mathcal{D}}$ is a strongly doubling sequence of order $(\beta, p)$ of non-negative definite matrices, and $\vec{f} \in \mathcal{S}^{\prime}/\mathcal{P}$.  Then 
\beq
\Vert \vec{f} \Vert_{\dot{F}^{\alpha q}_p ( \{ A_Q \} )} \approx \Vert \{ \langle \vec{f}, \ffi_{Q} \rangle \}_{Q \in \mathcal{D}} \Vert_{\dot{f}^{\alpha q}_p ( \{ A_Q \} )} .  \label{avgequiv}
\eeq 
Moreover, $\dot{F}^{\alpha q}_p ( \{ A_Q \} )$ is independent of the choice of $\varphi \in \mathcal{A}$, in the sense that the spaces defined for two such $\varphi$ are the same, with equivalent quasi-norms.
\end{Thm}

One direction of Theorem \ref{avgnormeq} is based on the following variation of the classical techniques involving the sampling theorem for functions of exponential type, as in, for example, \cite[p. 781]{FJ}.

\begin{Thm}\label{firstAQimbed}
Let $\ffi \in \mathcal{A}$.  Let $\tilde{\ffi} (x)= \overline{\ffi (-x)}$.  Suppose $\{ A_Q \}_{Q \in \mathcal{D}}$ is a weakly doubling sequence (of any order $r >0$) of non-negative definite matrices.  Then for $0<p<\infty$, $0 < q \leq \infty$, and $\alpha \in \R$, there exists $c$ depending on $\alpha, p, q, r, \varphi$ and the constant in (\ref{weakdoubest}) such that  for all $\vec{f} \in \cS^{'} / \cP (\R^n)$,
\beq
\left\Vert \left( \sum_{j \in \Z} \sum_{Q \in \mathcal{D}_j} \left( 2^{j\alpha} \sup_{x \in Q} \vert A_Q \ffi_j \ast \vec{f} (x)  \vert \chi_Q(x)  \right)^q \right)^{1/q} \right\Vert_{L^p (dx)} \leq c \Vert \vec{f} \Vert_{\dot{F}^{\alpha q}_p (\{ A_Q \} )} , \label{suptriebest}
\eeq
and 
\beq
\left\Vert  \left\{  \langle \vec{f}, \tilde{\ffi}_{Q} \rangle \right\} \right\Vert_{\dot{f}^{\alpha q}_p (\{ A_Q \})} \leq c \Vert \vec{f} \Vert_{\dot{F}^{\alpha q}_p (\{ A_Q \} )} . \label{vectortriebest}
\eeq
\end{Thm}

\begin{proof}
Let $\gamma \in \cS$ satisfy $\hat{\gamma} (\xi) = 1 $ for $|\xi| \leq 2$ and $\supp \hat{\gamma} \subseteq \{ \xi \in \R^n: |\xi| < \pi \}$.  Let $\gamma_j (x) = 2^{j n} \gamma (2^j x)$.  Then $\hat{\gamma}_j (\xi) = \hat{\gamma}(2^{- j} \xi)$.  Hence, for any $\vec{g} = (g_1, g_2, \dots , g_m)^T$ with $\supp \hat{g}_i \subseteq \{ \xi \in \R^n: |\xi| \leq 2^j  \}$ for $1 \leq i \leq m$, we have $\vec{g} = \vec{g} \ast \gamma_j$.  By \cite[Lemma 6.10]{FJW}, we have the identity
\[  \vec{g} (t) = \sum_{\ell \in \Z^n} \vec{g} (2^{- j} \ell) \,\,\, 2^{- j n} \gamma_j  ( t - 2^{- j} \ell) = \sum_{\ell \in \Z^n} \vec{g} (2^{- j} \ell) \,\,\,  \gamma  (2^j t - \ell)  . \]
We apply this identity with $\vec{g}(t) = \ffi_j \ast \vec{f} (t+ 2^{-j} y)$, for an arbitrary $ y \in \R^n$, to obtain
\[ \ffi_j \ast \vec{f} ( t + 2^{- j} y) = \sum_{\ell \in \Z^n} \ffi_j \ast \vec{f}  (2^{- j} \ell + 2^{- j}y ) \,\,\,  \gamma  (2^j t - \ell). \] 
We take $w \in Q_{00}$ and let $t = 2^{- j} k - 2^{- j} y + 2^{-j}w$, to obtain
\[ \ffi_j \ast \vec{f} ( 2^{- j} k + 2^{-j} w) = \sum_{\ell \in \Z^n} \ffi_j \ast \vec{f}  (2^{- j} \ell + 2^{- j}y ) \,\,\,  \gamma  (k - y +w - \ell),   \]
for $k \in \Z^n$.  Hence, 
\beq
A_{Q_{j k} } \ffi_j \ast \vec{f} ( 2^{- j} k + 2^{-j } w) = \sum_{\ell \in \Z^n} A_{Q_{j k}} \ffi_j \ast \vec{f}  (2^{- j} \ell + 2^{- j}y ) \,\,\,  \gamma  (k - y + w - \ell) . \label{vecsamplident}
\eeq
For $w, y \in Q_{00}$, we have $|\gamma (k - y+ w - \ell)| \leq c_{R} (1 + |k - \ell|)^{-R}$, for any $R>0$.  Pick $A$ with $0 < A \leq 1$ such that $p/A > 1$ and $q /A >1$.  Then by equation (\ref{vecsamplident}), 
\[ \sup_{x \in Q_{j k}} \left\vert A_{Q_{j k}}  \ffi_j \ast \vec{f} (x) \right\vert = \sup_{w \in Q_{00}}  \left|A_{Q_{j k}}  \ffi_j \ast \vec{f} ( 2^{- j} k + 2^{-j} w) \right| \]
\[ \leq c \sum_{\ell \in \Z^n} (1 + |k - \ell|)^{-R} \left| A_{Q_{j k}} \ffi_j \ast \vec{f}  (2^{- j} \ell + 2^{- j}y )  \right| . \]
The trivial imbedding $\ell^A \rightarrow \ell^1$ yields 
\[ \sup_{x \in Q_{j k}} \left\vert A_{Q_{j k}}  \ffi_j \ast \vec{f} (x) \right\vert^A  \leq c  \sum_{\ell \in \Z^n} (1 + |k - \ell|)^{-RA} \left|A_{Q_{j k}} \ffi_j \ast \vec{f}  (2^{- j} \ell + 2^{- j}y ) \right|^A  . \]
We average over $y \in Q_{00}$ to obtain 
\[  \sup_{x \in Q_{j k}} \left\vert A_{Q_{j k}}  \ffi_j \ast \vec{f} (x) \right \vert^A  \leq c  \sum_{\ell \in \Z^n} (1 + |k - \ell|)^{-RA} \int_{Q_{00}} \left|A_{Q_{j k}} \ffi_j \ast \vec{f}  (2^{- j} \ell + 2^{- j} y ) \right|^A \, dy \]
\[ = c \sum_{\ell \in \Z^n} (1 + |k - \ell|)^{-RA} 2^{j n} \int_{Q_{j \ell}} \left|A_{Q_{j k}} \ffi_j \ast \vec{f}  (s ) \right|^A \, ds . \]
By the weak doubling estimate (\ref{weakdoubest}),
\[  \left|A_{Q_{j k}} \ffi_j \ast \vec{f}  (s ) \right| \leq   \Vert A_{Q_{j k}} A_{Q_{j \ell}}^{-1} \Vert \left|A_{Q_{j \ell}} \ffi_j \ast \vec{f}  (s ) \right| \leq c ( 1 + |k - \ell|)^r \left|A_{Q_{j \ell}} \ffi_j \ast \vec{f}  (s ) \right| . \]
Therefore, we have
\beq
 \sup_{x \in Q_{j k}} \left\vert A_{Q_{j k}}  \ffi_j \ast \vec{f} (x) \right\vert^A  \leq c \sum_{\ell \in \Z^n} (1 + |k - \ell|)^{-A(R -r)} 2^{j n} \int_{Q_{j \ell}} \left|A_{Q_{j \ell}} \ffi_j \ast \vec{f}  (s ) \right|^A \, ds .  \label{estforlater}
\eeq
Thus,
\[ \sum_{Q \in \mathcal{D}_j } \left( 2^{j \alpha}  \sup_{ Q} \left\vert A_{Q}  \ffi_j \ast \vec{f}  \right\vert \chi_{Q} \right)^q
 = \sum_{k \in \Z^n} 2^{j\alpha q} \sup_{x \in Q_{jk}} \left\vert A_{Q_{j k}}  \ffi_j \ast \vec{f} (x)  \right\vert^q \chi_{Q_{j k}}(x) \]
\[ \leq c   \left| \sum_{k \in \Z^n} \sum_{\ell \in \Z^n} (1 + |k - \ell|)^{-A(R-r)} 2^{j n} \int_{Q_{j \ell}} \left| 2^{j \alpha} A_{Q_{j \ell}} \ffi_j \ast \vec{f}  (s ) \right|^A \, ds   \chi_{Q_{j k}} \right|^{q/A}, \]
where in the last step we used the disjointness of the cubes $Q_{j k}$ for $k \in \Z^n$ to take the exponent $q/A$ outside the sum on $k$.
We claim that for any locally integrable function $h$, 
\beq
\sum_{k \in \Z^n} \sum_{\ell \in \Z^n} (1 + |k - \ell|)^{-A(R-r)} 2^{j n} \int_{Q_{j \ell}} | h(s)| \, ds   \chi_{Q_{j k} }  \leq c M( h)  ,  \label{newmaxineq}
\eeq
where $M$ is the Hardy-Littlewood maximal function, if we choose $A(R-r)>2n$, which we may.  Assuming inequality \eqref{newmaxineq} momentarily, and applying it above with $\ds h = \sum_{Q \in \mathcal{D}_j}  \left( 2^{j \al} \vert A_{Q} \ffi_j \ast \vec{f} \vert \chi_Q \right)^A $, we obtain
\[ \sum_{Q \in \mathcal{D}_j } \left( 2^{j \al} \sup_{Q} \left\vert A_{Q}  \ffi_j \ast \vec{f}  \right\vert \chi_{Q} \right)^q \leq c \left(  M \left( \sum_{Q \in \mathcal{D}_j}  \left( 2^{j \al} \vert A_{Q} \ffi_j \ast \vec{f} \vert \chi_Q \right)^A  \right) \right)^{q/A}  . \]
Then 
\[  \left\Vert \left(  \sum_{j \in \Z} \sum_{Q \in \mathcal{D}_j } \left( 2^{j \al } \sup_Q \left\vert A_{Q}  \ffi_j \ast \vec{f}  \right\vert \chi_{Q} \right)^q  \right)^{1/q} \right\Vert_{L^p (\R^n)} \]
\[ \leq c \left\Vert \left(  \sum_{j\in \Z} \left(  M \left( \sum_{Q \in \mathcal{D}_j}  \left( 2^{j \al} \vert A_{Q} \ffi_j \ast \vec{f} \vert \chi_Q \right)^A  \right)  \right)^{q/A} \right)^{A/q} \right\Vert_{L^{p/A} (\R^n)}^{1/A} . \]
Applying the Fefferman-Stein vector-valued maximal inequality (\cite{FS}) with indices $p/A, q/A >1$, we remove $M$ and untangle the indices to obtain
(\ref{suptriebest}). 

It remains to prove (\ref{newmaxineq}).  For a fixed $x$, let $Q_{j k}$ be the dyadic cube of length $2^{-j}$ containing $x$.  Let $B_{\ell}$ be the smallest ball containing $x$ and the cube $Q_{j \ell}$.  The radius of $B_{\ell}$ is equivalent to $2^{- j} (1 + |k-\ell|)$.  Hence, 
\[ \int_{Q_{j \ell}} | h(s) | \, ds \leq \int_{B_{\ell}} |h(s)| \, ds \leq c 2^{-j n} (1 + |k - \ell|)^n M (h)(x) . \]
For this $x$, the left side of (\ref{newmaxineq}) is
\[ \sum_{\ell \in \Z^n} (1 + |k - \ell|)^{-A(R-r)} 2^{j n} \int_{Q_{j \ell}} | h(s)| \, ds \]
\[ \leq  c \sum_{\ell \in \Z^n} (1 + |k - \ell|)^{-A(R-r) +n } M (h)(x) \leq c M (h)(x), \]
since we have chosen $A(R-r)-n>n$.

Finally, \eqref{vectortriebest} follows from \eqref{suptriebest} because $|Q_{j k}|^{-1/2} \langle \vec{f}, \tilde{\ffi}_{Q_{j k}} \rangle = \ffi_j \ast \vec{f} (x_{Q_{jk}})$, so 
\[ \Vert  \left\{  \langle f, \tilde{\ffi}_{Q} \rangle \right\}_{Q \in \mathcal{Q}} \Vert_{\dot{f}^{\alpha q}_p (\{ A_Q \})} = \left\Vert \left(\sum_{j \in \Z}  \sum_{Q \in \mathcal{D}_j} \left( |Q|^{-\alpha/n }  \vert A_Q \ffi_j \ast \vec{f} (x_Q)  \vert \chi_Q  \right)^q \right)^{1/q} \right\Vert_{L^p }, \]
which is obviously dominated by the left side of \eqref{suptriebest}.
\end{proof}

Heading toward an estimate converse to \eqref{vectortriebest}, we first introduce almost diagonal matrices.  

\begin{Def} \label{Defalmostdiag}
Let $0<p<\infty, 0<q\leq \infty, \alpha \in \R$, and $\beta>0$.  A matrix $B= \{ b_{QP} \}_{Q, P \in \mathcal{D}} $ is almost diagonal, written $B \in {\bf ad}^{\alpha, q}_p (\beta)$, if there exists $C>0$ such that $|b_{QP}| \leq C \omega_{QP}$ for all $Q, P \in \mathcal{D}$, where
\[ \omega_{QP} = \min \left\{ \left( \frac{\ell(P)}{\ell(Q)} \right)^{\alpha_1} , \left( \frac{\ell(Q)}{\ell(P)} \right)^{\alpha_2}   \right\} \left( 1 + \frac{|x_Q-x_P|}{\max (\ell(Q), \ell(P))}  \right)^{-R}, \]
for some $\alpha_1 > -\alpha -\frac{n}{2} + \frac{\beta -n}{p} + \frac{n}{\min(1, p, q)} , \alpha_2 > \alpha + \frac{n}{2} + \frac{n}{p}$, and $R > \frac{n}{\min(1, p, q)} + \frac{\beta}{p}$.  
\end{Def}

A matrix $B= \{ b_{QP} \}_{Q, P \in \mathcal{D}} $ acts on a sequence $\vec{s}= \{ \vec{s}_Q \}_{Q \in \mathcal{D}}$ by matrix multiplication in each component:  $B\vec{s}=\vec{t} = \{  \vec{t}_Q \}_{Q \in \mathcal{D}}$, where $\vec{t}_Q = \sum_{P \in \mathcal{Q}} b_{QP} \vec{s}_P$, if that series converges absolutely for all $Q$.  The following result can be reduced to the classical case using (\ref{reduction}).  

\begin{Thm} \label{ADmatrixbnded}
Let $0<p<\infty, 0<q\leq \infty$, $\alpha \in \R$, and $\beta>0$. Suppose $\{ A_Q \}_{Q \in \mathcal{D}}$ is a sequence of non-negative definite matrices, which is strongly doubling of order $(\beta, p)$ for some $\beta>0$.  Suppose $B \in {\bf ad}^{\alpha, q}_p (\beta)$.  Then $B$ defines a bounded operator on $\ds \dot{f}^{\alpha q}_p(\{A_Q\})$.
\end{Thm}

\begin{proof} Define $\vec{t} = B \vec{s}$ as above, for $\vec{s} \in \ds \dot{f}^{\alpha q}_p(\{A_Q\})$, and $B= \{ b_{QP} \}_{Q, P \in \mathcal{D}}$.  To employ (\ref{reduction}), define a scalar sequence $t_A = \{ t_{A,Q} \}_{Q \in \mathcal{D} } $, where $t_{A, Q} = |A_Q \vec{t}_Q|$, and similarly define $s_A$.  Then $\Vert \vec{t} \, \Vert_{\dot{f}^{\alpha q}_p(\{A_Q \})} = \Vert t_A \, \Vert_{\dot{f}^{\alpha q}_p} $ and similarly for $\vec{s}$, where $\dot{f}^{\alpha q}_p$ is the scalar, unweighted space as in \cite{FJ2}.  Let $\gamma_{QP} = \omega_{QP} \Vert A_Q A_P^{-1} \Vert$. Then 
\[ t_{A,Q} = |A_Q \vec{t}_Q| = \left| A_Q \sum_{P \in \mathcal{D}} b_{QP} \vec{s}_P \right| \leq \sum_{P \in \mathcal{D}} |b_{QP}| |A_Q \vec{s}_P| \]
\[ \leq C \sum_{P \in \mathcal{D}} \omega_{QP} \Vert A_Q A_P^{-1} \Vert |A_P \vec{s}_P |= C \sum_{P \in \mathcal{D}} \gamma_{QP} s_{A,P} .  \]
That is, if $G = \{ \gamma_{QP} \}$, then $t_{A,Q} \leq C (G (s_A))_Q $ for each $Q \in \mathcal{D}$. By (\ref{stdoubest}), $\gamma_{QP}$ satisfies the scalar, unweighted almost diagonality condition (3.1) in \cite{FJ2}.  Thus, by Theorem 3.3 in \cite{FJ2}, $G$ is bounded on $\dot{f}^{\alpha q}_p$.  Therefore,  
\[ \Vert \vec{t} \, \Vert_{\dot{f}^{\alpha q}_p(\{A_Q \})} = \Vert t_A \, \Vert_{\dot{f}^{\alpha q}_p}  \leq  C  \Vert s_A \, \Vert_{\dot{f}^{\alpha q}_p}= C \Vert \vec{s} \, \Vert_{\dot{f}^{\alpha q}_p(\{A_Q \})} . \] 
\end{proof} 

We need the notion of smooth molecules, as in \cite{FJ2} or \cite[Section 5]{R1}. Unlike the case of $\varphi_Q$ or $\psi_Q$ in \eqref{defphiQ}, the notation $m_Q$ in the following definition is not meant to imply that each $m_Q$ is obtained from a fixed $m$ by translation and dilation; here, $Q$ is merely an index.  

\begin{Def}  Let $0<\de \leq 1$, $M>0$ and $N, K \in \Z$.  We say $\{m_Q \}_{Q \in \mathcal{D}}$ is a family of  smooth
$(N, K, M, \delta)$-molecules if there exists $\epsilon>0$ and $C>0$ such that, for all $Q \in \mathcal{D}$, 
\begin{enumerate}
\item[(M1)]
$\ds \int x^{\ga} m_Q (x) \, dx = 0, ~~\mbox{for} ~~|\ga| \leq N$,

\item[(M2)]
$\ds |m_Q(x)| \leq C |Q|^{-1/2} \left(1+ \frac{|x-x_Q|}{l(Q)}\right)^{-\max(M,N + 1 +n+ \epsilon )},$

\item[(M3)]
$\ds |D^{\ga} m_Q(x)| \leq C |Q|^{-1/2-|\ga|/n} \left(1+ \frac{|x-x_Q|}{l(Q)}\right)^{-M}
\mbox{if} ~~|\ga| \leq K,$

\item[(M4)]
$\ds |D^{\ga} m_Q(x) - D^{\ga} m_Q(y)| \leq C |Q|^{-\frac12-\frac{|\ga|}{n}-\frac{\de}n}
|x-y|^{\de}$\\

$\ds \times \sup_{|z|\leq|x-y|} \left(1+ \frac{|x-z-x_Q|}{l(Q)}\right)^{-M}\,
\mbox{if} ~|\ga|=K.$
\end{enumerate}
It is understood that (M1) is void if $N<0$ and (M3), (M4) are void if $K<0$.
\end{Def}

We need the following estimates from \cite[Appendix B]{FJ2}.

\begin{Lemma}\label{convest}
Suppose $\ffi \in \mathcal{A}$ and $\{m_Q \}_{Q \in \mathcal{D}}$ is a family of  smooth $(N, K, M, \delta)$-molecules.  Then there exists $c>0$ such that 

\vspace{0.1in}

\noindent{(i)}: for all $P \in \mathcal{D}$ with $\ell(P) = 2^{-k} \geq 2^{-j}$, we have
\beq
| \ffi_{j} \ast m_P (x)|  \leq c 2^{kn/2}  2^{-(j - k) (K + \delta)} \left(1 + 2^k|x-x_P| \right)^{-M}, \label{muleqnuest}
\eeq 
\noindent{and}

\vspace{0.1in}

\noindent{(ii)}: for all $P \in \mathcal{D}$ with $\ell(P) = 2^{-k} \leq 2^{-j}$, we have
\beq
| \ffi_{j} \ast m_P (x)| \leq c 2^{kn/2} 2^{-(k - j) (N+1+n)} \left(1 + 2^j |x-x_P|  \right)^{-M}   . \label{mugeqnuest}
\eeq
\end{Lemma}

\vspace{0.1in}

\begin{Thm}\label{thirdAQest} Let $0<p<\infty$, $0 < q \leq \infty$, $\alpha \in \R$.  Suppose $\{ A_Q \}_{Q \in \mathcal{D}}$ is a strongly doubling sequence of order $(\beta, p)$ of non-negative definite matrices.  Suppose $N  \in \Z$, $K \in \Z$, $M>0$ and $\delta \in (0,1]$ satisfy $N > -\alpha + \frac{\beta -n}{p}  + \frac{n}{\min(1,p,q)}  -n -1, K + \delta > \alpha + \frac{n}{p}$, and $M > \frac{n}{\min(1,p,q)} + \frac{\beta}{p}$.  Suppose $\{m_Q \}_{Q \in \mathcal{D}}$ is a family of  smooth $(N, K, M, \delta)$-molecules.  Suppose $\vec{s} = \{ \vec{s}_Q \}_{Q \in \mathcal{D}} \in \dot{f}^{\alpha q}_p (\{ A_Q \})$.  Then $\vec{f} = \sum_{Q \in \mathcal{D}} \vec{s}_Q m_Q \in \dot{F}^{\alpha q}_p (\{ A_Q \} )$ and 
\[ \Vert \vec{f} \Vert_{\dot{F}^{\alpha q}_p (\{ A_Q \} )} \leq c \Vert   \vec{s} \Vert_{\dot{f}^{\alpha q}_p (\{ A_Q \})} . \]
In particular, for $\varphi \in \mathcal{A}$, we have 
\begin{equation}
\Vert \vec{f} \Vert_{\dot{F}^{\alpha q}_p (\{ A_Q \} )} \leq c \Vert  \{ \langle \vec{f}, \varphi_Q \rangle \}_{Q \in \mathcal{D}} \Vert_{\dot{f}^{\alpha q}_p (\{ A_Q \})} . \label{convtriebelest}
\end{equation}
\end{Thm}

\begin{proof}
For $Q \in \mathcal{D}_j$, let $g_Q = |Q|^{1/2} | A_Q \varphi_j \ast \sum_{P \in \mathcal{D} } \vec{s}_P m_P | $, so that 
\[  \Vert \vec{f} \Vert_{\dot{F}^{\alpha q}_p ( \{ A_Q \} )} = \left\Vert \left(   \sum_{j \in \Z} \sum_{Q \in \mathcal{D}_j} \left( |Q|^{-\alpha/n -1/2} g_Q \chi_Q   \right)^q  \right)^{1/q} \right\Vert_{L^p(\R^n)} .\]
Note that for any $P, Q \in \mathcal{D}$ and $x \in Q$, 
\[ 1 + \frac{|x - x_P|}{\max\{\ell(P), \ell(Q)\}} \approx 1 + \frac{|x_Q - x_P|}{\max\{\ell(P), \ell(Q)\}} . \]
Hence, by Lemma \ref{convest}, $|Q|^{1/2} |\varphi_j \ast m_P(x)| \leq C \omega_{QP}$, for all $x \in Q$, where $\omega_{QP}$ is as in Definition \ref{Defalmostdiag}.  Therefore, 
\[ g_Q \chi_Q \leq \sum_{P \in \mathcal{D}} |Q|^{1/2} |  \varphi_j \ast  m_P | | A_Q \vec{s}_P| \chi_Q  \leq C \sum_{P \in \mathcal{D}} \omega_{QP} \Vert A_Q A_P^{-1} \Vert |A_P \vec{s}_P| \chi_Q  .\]
Let $G$ and $s_A$ be defined as in the proof of Theorem \ref{ADmatrixbnded}.  Then $G$ is bounded on the scalar, unweighted space $\dot{f}^{\alpha q}_p$.  Substituting $g_Q \chi_Q \leq C (G (s_A))_Q \chi_Q$ above gives
\[  \Vert \vec{f} \Vert_{\dot{F}^{\alpha q}_p ( \{ A_Q \} )}  \leq C \Vert G(s_A) \Vert_{\dot{f}^{\alpha q}_p} \leq C \Vert s_A \Vert_{\dot{f}^{\alpha q}_p} =  C \Vert \vec{s} \Vert_{\dot{f}^{\alpha q}_p ( \{ A_Q \} )}  . \]
Then \eqref{convtriebelest} follows since $\vec{f} = \sum_{Q \in \mathcal{D}} \langle \vec{f}, \varphi_Q \rangle \psi_Q$ by \eqref{phitraniden}, and $\{ \psi_Q \}_{Q \in \mathcal{D}}$ is a family of smooth $(N,K M, \delta)$ molecules for any possible $N,K,M$, and $\delta$.  
\end{proof}

\begin{proofof} Theorem \ref{avgnormeq}.  
We first prove the independence of the spaces on the choice of test function $\varphi \in \mathcal{A}$. Suppose $\varphi, \gamma \in \mathcal{A}$.  For the duration of this proof, we label spaces defined by $\varphi$ as $\dot{F}^{\alpha q}_p ( \{ A_Q \}, \varphi )$, and similarly for $\gamma$.   We can select $\psi, \tau \in \mathcal{A}$ such that $\sum_{j \in \Z} \overline{\widehat{\varphi_j}}(\xi) \widehat{\psi_j} (\xi)  =1$ and $\sum_{j \in \Z} \overline{\widehat{\gamma_j}}(\xi) \widehat{\tau_j} (\xi)  =1$ for all $\xi \neq 0$.  Define $\vec{\tilde{s}} = \{ \vec{\tilde{s}}_Q \}_{Q \in \mathcal{D}}$ by $\vec{\tilde{s}}_Q = \langle \vec{f}, \tilde{\varphi}_Q \rangle$ and $\vec{t} = \{ \vec{t}_Q \}_{Q \in \mathcal{D}}$ by $\vec{t}_Q = \langle \vec{f}, \gamma_Q \rangle$. We have $ \Vert \vec{f} \Vert_{\dot{F}^{\alpha q}_p ( \{ A_Q \}, \gamma )} \leq C \Vert \vec{t} \Vert_{\dot{f}^{\alpha q}_p ( \{ A_Q \})}$ by \eqref{convtriebelest} with $\gamma$ in place of $\ffi$. Notice that $\sum_{j \in \Z} \overline{\widehat{\tilde{\psi}_j}}(\xi) \widehat{\tilde{\varphi}_j} (\xi) = \sum_{j \in \Z} \overline{\widehat{\varphi_j}}(\xi) \widehat{\psi_j} (\xi)  = 1$ for all $\xi \neq 0$.  So, applying \eqref{phitraniden} with $\varphi, \psi$, and $\vec{f}$ replaced by $\tilde{\psi}, \tilde{\varphi}$, and $\gamma_Q$, respectively, we have $\gamma_Q = \sum_{P \in \mathcal{D}} \langle \gamma_Q, \tilde{\psi}_P \rangle \tilde{\varphi}_P$.  Note that $\gamma_Q \in \cS_0$, so $\sum_{P \in \mathcal{D}} \langle \gamma_Q, \tilde{\psi}_P \rangle \tilde{\varphi}_P$ converges in $\cS$. Therefore, since $\vec{f} \in \cS'/\mathcal{P} (\R^n)$,
\[ \vec{t}_Q = \langle \vec{f}, \sum_{P \in \mathcal{D}} \langle \gamma_Q, \tilde{\psi}_P \rangle \tilde{\varphi}_P \rangle =  \sum_{P \in \mathcal{D}} \langle \gamma_Q, \tilde{\psi}_P \rangle \tilde{s}_P .   \]
Notice that, for $\ell(Q) = 2^{-j}$,  $\langle \gamma_Q, \tilde{\psi}_P \rangle = |Q|^{1/2} \gamma_j \ast \psi_P (x_Q)$.  Since $\{ \psi_Q \}_{Q \in \mathcal{D}}$ is a family of smooth $(N,K, M, \delta)$-molecules for all possible $N,K,M$, and $\delta$, Lemma \ref{convest} implies that the matrix $B = \{ b_{QP} \}_{Q, P \in \mathcal{D}}$ defined by $b_{QP} = \langle \gamma_Q, \tilde{\psi}_P \rangle$ is almost diagonal, i.e., $B \in \bf{ad}^{\alpha, q}_p (\beta)$, for all possible $\alpha, q, p,$ and $\beta$.  By Theorem \ref{ADmatrixbnded}, $B$ is bounded on $\dot{f}^{\alpha q}_p ( \{ A_Q \} )$. Thus,
\[ \Vert \vec{t} \Vert_{\dot{f}^{\alpha q}_p ( \{ A_Q \})}  \leq C  \Vert B\vec{\tilde{s}} \Vert_{\dot{f}^{\alpha q}_p ( \{ A_Q \})}\leq   C  \Vert \vec{\tilde{s}} \Vert_{\dot{f}^{\alpha q}_p ( \{ A_Q \})}  \leq C \Vert \vec{f} \Vert_{\dot{F}^{\alpha q}_p ( \{ A_Q \}, \varphi )} ,\]
where the last step is by Theorem \ref{firstAQimbed}.  Hence, we have $\Vert \vec{f} \Vert_{\dot{F}^{\alpha q}_p ( \{ A_Q \}, \gamma )} \leq C  \Vert \vec{f} \Vert_{\dot{F}^{\alpha q}_p ( \{ A_Q \}, \ffi )} $, which implies equivalence by interchanging $\gamma$ and $\varphi$.

To prove \eqref{avgequiv}, first apply Theorem \ref{firstAQimbed} with $\varphi$ replaced by $\tilde{\ffi}$ to obtain $\Vert \{ \langle \vec{f}, \ffi_{Q} \rangle \} \Vert_{\dot{f}^{\alpha q}_p ( \{ A_Q \} )} \leq c \Vert \vec{f} \Vert_{\dot{F}^{\alpha q}_p ( \{ A_Q \}, \tilde{\ffi} )} $.  For $\varphi \in \mathcal{A}$, we have $\tilde{\ffi} \in \mathcal{A}$, so we have just proved that the last norm is equivalent to the one with $\ffi$ in place of $\tilde{\ffi}$.  Then applying \eqref{convtriebelest} completes the proof. 
\end{proofof}

\begin{Thm}\label{avgwavechar}
Let $0 < p < \infty$, $0 < q \leq \infty$, and $\alpha \in \R$.  Suppose $\{ A_Q \}_{Q \in \mathcal{Q}}$ is a strongly doubling sequence of order $(\beta, p)$ of non-negative definite matrices.   Suppose that for some sufficiently large positive numbers $N_0, R$, and $S$ (depending on $p, q, \alpha, n$, and $\beta$), the generators $\{ \psi^{(i)} \}_{1 \leq i \leq 2^n -1} $ of a wavelet basis satisfy $\int_{\R^n} x^{\gamma} \psi^{(i)} (x) \, dx =0$ for all multi-indices $\gamma$ with $|\gamma| \leq N_0$ and $|D^{\gamma} \psi^{(i)} (x)| \leq C (1+|x|)^{-R}$ for all $|\gamma| \leq S$.  Then 
\beq \label{avgwaveequiv}
\Vert \vec{f} \Vert_{\dot{F}^{\alpha q}_p ( \{A_Q\}} \approx \sum_{i=1}^{2^n -1} \Vert \{ \langle \vec{f}, \psi^{(i)}_Q \rangle \}_{Q \in \mathcal{D}} \Vert_{\dot{f}^{\alpha q}_p ( \{A_Q\} )}  . 
\eeq
\end{Thm}

\begin{proof}
If $N_0 \geq N$, $S \geq K+ \delta$, and $R > \max(M, N+1+n)$, then $\{  \psi_{Q,i} \}_{Q \in \mathcal{D}}$ is a family of smooth $(N,K,M,\delta)$ molecules for each $i$, so the estimate $\Vert \vec{f} \Vert_{\dot{F}^{\alpha q}_p ( \{ A_Q \} )} \leq c \sum_{i=1}^{2^n -1} \Vert \{ \langle \vec{f}, \psi^{(i)}_{Q} \rangle \}_{Q \in \mathcal{D}} \Vert_{\dot{f}^{\alpha q}_p ( \{ A_Q \} )}$ follows from Theorem \ref{thirdAQest} and the wavelet identity $\vec{f} = \sum_{Q,i} \langle \vec{f}, \psi^{(i)}_Q \rangle \psi^{(i)}_Q$. The proof of the converse estimate is similar to the proof of Theorem \ref{avgnormeq}.  For each $i$, define $\vec{s}^{\,(i)}= \{ \vec{s}^{\,(i)}_Q  \}_{Q \in \mathcal{D}}$ by $\vec{s}^{\,(i)}_Q = \langle \vec{f} , \psi^{(i)}_Q \rangle$.  Using (\ref{phitraniden}), we have
\[  \vec{s}^{\,(i)}_Q = \langle \vec{f} , \sum_{P \in \mathcal{D}} \langle  \psi^{(i)}_Q    , \ffi_P  \rangle \psi_P \rangle  = \sum_{P\in \mathcal{D}} b_{QP}^{(i)} \vec{s}_P = (B^{(i)} \vec{s})_Q,\]
where $B^{(i)}= \{ b^{(i)}_{QP} \}_{Q,P \in \mathcal{D}} $ is defined by $b^{(i)}_{QP} = \langle \psi_Q^{(i)} , \ffi_P \rangle $ and $\vec{s} = \{ \vec{s}_Q \}_{Q \in \mathcal{D}} $ is defined by $\vec{s}_Q = \langle \vec{f}, \psi_Q \rangle$.  Since $\psi \in \mathcal{A}$, Theorem \ref{avgnormeq} gives the equivalence $ \Vert  \vec{s} \Vert_{\dot{f}^{\alpha q}_p ( \{ A_Q \} )} \approx   \Vert \vec{f} \Vert_{\dot{F}^{\alpha q}_p ( \{ A_Q \} )}  $.  We claim that for $N_0, R$, and $S$ sufficiently large, $B^{(i)} \in \bf{ad}^{\alpha,q}_p (\beta)$, and thus, $B^{(i)}$ is bounded on $\dot{f}^{\alpha q}_p ( \{ A_Q \} )$ by Theorem \ref{ADmatrixbnded}.  Assuming this claim, we have 
\[ \Vert \vec{s}^{\,(i)} \Vert_{\dot{f}^{\alpha q}_p ( \{A_Q\} )}  =  \Vert B^{(i)} \vec{s} \Vert_{\dot{f}^{\alpha q}_p ( \{A_Q\} )} \leq c \Vert  \vec{s} \Vert_{\dot{f}^{\alpha q}_p ( \{A_Q\} )} \leq c   \Vert \vec{f} \Vert_{\dot{F}^{\alpha q}_p ( \{ A_Q \} )}   , \]
yielding (\ref{avgwaveequiv}).  To show that $B^{(i)} \in \bf{ad}^{\alpha,q}_p (\beta)$, note that for $\ell(P) = 2^{-j}$, we have $b^{(i)}_{QP} = 2^{-jn/2} \tilde{\ffi}_j \ast \psi^{(i)}_Q (x_P)$.  Applying Lemma \ref{convest} with $P$ replaced by $Q$ and $\tilde{\ffi} \in \mathcal{A}$ in place of $\ffi$, we see that $B^{(i)} \in \bf{ad}^{\alpha,q}_p (\beta)$ if $\{ \psi^{(i)}_Q  \}_{Q \in \mathcal{D}}$ is a family of smooth $(N,K, M, \delta)$-molecules for some $N> N_1= \alpha -1 + n/p, K + \delta> S_1= -\alpha + \frac{n}{\min(1,p,q)} -n + \frac{\beta -n}{p}$, and $M> M_1 = \frac{n}{\min(1,p,q)} + \frac{\beta}{p}$, which in turn holds if $\psi^{(i)}$ satisfies the conditions in the statement of the theorem for $N_0>   N_1, S >S_1$, and $R> R_1$.
\end{proof}

\section{Equivalence of the averaging and non-averaging spaces}\label{relavgnonavg}

Although the results in Section \ref{normeqavg} required only the strong doubling condition on $\{ A_Q\}$, we now assume the $A_p$ condition on $W$ to obtain the equivalence between the weighted sequence and function spaces and their averaged counterparts, as follows.  

\begin{Thm}\label{nonavnormeq} Suppose $0 < p < \infty, 0 < q \leq \infty, \alpha \in \R$ and $\ffi \in \mathcal{A}$.  Suppose $W \in A_p$, and $\{ A_Q \}_{Q \in \mathcal{Q}}$ is a sequence of reducing operators of order $p$ for $W$.  
Then for any sequence $\vec{s} = \{ \vec{s}_Q \}_{Q \in \mathcal{D}}$, 
\[ \Vert \vec{s} \Vert_{\dot{f}^{\alpha q}_p (W)} \approx \Vert \vec{s} \Vert_{\dot{f}^{\alpha q}_p (\{A_Q \})}  \]
and, for any $\vec{f} \in \mathcal{S}^{\prime}/ \mathcal{P}$, 
\[ \Vert \vec{f} \Vert_{\dot{F}^{\alpha q}_p (W)} \approx \Vert \vec{f} \Vert_{\dot{F}^{\alpha q}_p (\{A_Q \})} . \]
\end{Thm}

We build up to the proof of Theorem \ref{nonavnormeq} by first discussing some of the consequences of the $A_p$ condition. We will use the following results from \cite{G}, pp. 207-8 and p. 210; see \cite{CG} for $p=2$.

\begin{Lemma} \label{GoldbergLemma}
Suppose $1<p< \infty, p^{\, \prime} = p/(p-1), W \in A_p$, and $\{ A_Q \}_{Q \in \mathcal{D}}$ is a sequence of reducing operators of order $p$ for $W$.  Then there exists $\delta >0$ (depending on $W$) and constants $C_r >0$ such that
\beq \label{equivApcond}
\sup_{Q \in \mathcal{D}} \frac{1}{|Q|} \int_Q \Vert A_Q W^{-1/p} (x) \Vert^r \, dx \leq C_r  \,\,\, \mbox{for} \,\,\, r < p^{\, \prime} + \delta, 
\eeq
\begin{equation} \label{revholder}
\sup_{Q \in \mathcal{D}}  \frac{1}{|Q|} \int_Q \Vert W^{1/p} (x) A_Q^{-1} \Vert^{r} \, dx  \leq C_r \,\,\, \mbox{for} \,\,\, r < p + \delta,  
\end{equation}
and 
\begin{equation} \label{chgolineq}
\sup_{Q \in \mathcal{D}} \frac{1}{|Q|} \int_Q   \sup_{P \in \mathcal{D}: x \in  P \subseteq Q} \Vert W^{1/p}(x) A_P^{-1} \Vert^{r} \, dx \leq C_r \,\,\, \mbox{for} \,\,\, r < p + \delta.    
\end{equation}
\end{Lemma}

We need the following analogue of Lemma \ref{GoldbergLemma} for $0<p\leq 1$.  

\begin{Lemma} \label{GoldbergLemmaple1}
Suppose $0<p\leq 1, W \in A_p$, and $\{ A_Q \}_{Q \in \mathcal{D}}$ is a sequence of reducing operators of order $p$ for $W$.  Then 
\beq \label{equivApcondp<1}
\sup_{Q \in \mathcal{D}} \esssup_{x \in Q}  \Vert A_Q W^{-1/p} (x) \Vert  < \infty, 
\eeq
and \eqref{revholder}, \eqref{chgolineq} hold for some $\delta >0$.
\end{Lemma}

\begin{proof}  We use the fact that $\Vert B \Vert \leq c_m \sum_{i=1}^m |B e_i |$ for any $m \times m $ matrix $B$, where $\{e_i \}_{i=1}^m$ are the standard unit Euclidean basis vectors in $\C^m$.  To prove (\ref{equivApcondp<1}), note that for a.e. $x \in Q$ we have 
\[   \vert A_Q W^{-1/p} (x) e_i \vert  
\leq c  \left( \frac{1}{|Q|} \int_Q  \vert W^{1/p} (y) W^{-1/p}(x) e_i \vert^p \, dy \right)^{1/p} \]
\[\leq c \left( \frac{1}{|Q|} \int_Q  \Vert W^{1/p} (y) W^{-1/p}(x)  \Vert^p \, dy \right)^{1/p} \leq c, \]
by definition (\ref{defAp2}).   

To prove (\ref{revholder}), the assumption $W \in A_p$ implies that for all $\vec{y} \in \C^m$, the scalar weights $w_{\vec{y}} (x) = |W^{1/p} (x) \vec{y}|^p$ are uniformly in $A_1$, by \cite{FR}, Lemma 2.1.  Hence (see e.g., \cite{Gr}, Theorem 9.2.2), they satisfy a uniform reverse H\"{o}lder condition: there exists $\gamma >0$ such that 
\begin{equation} \label{firstrevholder} 
 \left( \frac{1}{|Q|} \int_Q \vert W^{1/p} (x) \vec{y} |^{p(1+\gamma)} \right)^{1/(1 + \gamma)} \, dx \leq c  \frac{1}{|Q|} \int_Q \vert W^{1/p} (x) \vec{y} |^{p} \, dx, 
\end{equation}  
with $c$ independent of $\vec{y}$ and $Q$.  Applying (\ref{firstrevholder}) with $\vec{y} = A_Q^{-1} e_i$,  
\[ \frac{1}{|Q|} \int_Q \Vert W^{1/p} (x) A_Q ^{-1} \Vert^{p(1+\gamma)} \, dx \leq c \frac{1}{|Q|} \int_Q \left( \sum_{i=1}^m \vert W^{1/p} (x) A_Q ^{-1} e_i \vert \right)^{p(1+\gamma)} \, dx \]
\[ \leq c \sum_{i=1}^m \frac{1}{|Q|} \int_Q \vert W^{1/p} (x) A_Q ^{-1} e_i \vert^{p(1+\gamma)} \, dx \leq  c  \sum_{i=1}^m  \left( \frac{1}{|Q|} \int_Q  \vert W^{1/p} (x) A_Q ^{-1} e_i \vert^p \, dx \right)^{(1+ \gamma)} \]
\[ = c \sum_{i=1}^m  \left( \vert A_Q A_Q^{-1} e_i \vert \right)^{p(1+ \gamma)} \leq  c .  \]
Letting $\delta= p\gamma$, we have (\ref{revholder}).

For (\ref{chgolineq}), we use the fact that, for $p\leq 1$,  
\[ \vert A_Q^{-1} \vec{y} \vert \approx  \esssup_{x \in Q} \vert W^{-1/p} (x) \vec{y} \vert    ,\]
with equivalence constants independent of $Q$ and $\vec{y}$, by \cite{FR}, Lemma 5.4. Recall that $\Vert AB\Vert = \Vert BA \Vert$ for any self-adjoint $A$ and $B$.  Therefore, for $P,Q \in \mathcal{D}$ with $P \subseteq Q$, we have
\[ \Vert A_Q A_P^{-1} \Vert \leq c \sum_{i =1}^m \vert A_P^{-1} A_Q e_i \vert \leq c \sum_{i =1}^m \esssup_{x \in P} \vert  W^{-1/p}(x) A_Q e_i \vert  \]
\[ \leq c \esssup_{ x \in P} \Vert W^{-1/p}(x) A_Q \Vert = c \esssup_{ x \in Q} \Vert A_Q W^{-1/p}(x)  \Vert \leq c, \]
by (\ref{equivApcondp<1}).  Hence, for a.e. $x \in P$, 
\[ \Vert W^{1/p}(x) A_P^{-1} \Vert \leq \Vert W^{1/p}(x) A_Q^{-1} \Vert \Vert A_Q A_P^{-1} \Vert \leq c \Vert W^{1/p}(x) A_Q^{-1} \Vert .  \]
Thus, $\sup_{P \in \mathcal{D}: x \in  P \subseteq Q} \Vert W^{1/p}(x) A_P^{-1} \Vert \leq c \Vert W^{1/p}(x) A_Q^{-1} \Vert$ a.e., so (\ref{chgolineq}) follows from (\ref{revholder}).
\end{proof}

\begin{Thm}\label{AQ-Wseqspimbed}
Suppose $\alpha \in \R, 0 < p < \infty, 0 < q \leq \infty, W \in A_p$, and $\{ A_Q \}_{Q \in \mathcal{D}}$ is a sequence of reducing operators of order $p$ for $W$. Then there exists $c>0$ such that for all $\vec{s} = \{ \vec{s}_Q \}_{Q \in \mathcal{Q}}$, 
\beq \label{AQWseqimbed}
\Vert \vec{s} \Vert_{\dot{f}^{\alpha q}_p ( \{ A_Q \} )} \leq c \Vert \vec{s} \Vert_{\dot{f}^{\alpha q}_p ( W )} .  
\eeq 
\end{Thm}

\begin{proof}
For $0<p\leq 1$, inequality (\ref{equivApcondp<1}) implies that 
\[ |A_Q \vec{s}_Q| \chi_{Q} \leq  \Vert A_Q W^{-1/p} \Vert | W^{1/p} \vec{s}_Q| \chi_{Q} \leq  c_p  | W^{1/p} \vec{s}_Q| \chi_{Q} \,\, a.e., \]
which implies (\ref{AQWseqimbed}). 

Now suppose $p>1$.  Let $C_1$ be the constant from (\ref{equivApcond}) when $r=1$.  For each $Q \in \mathcal{D}$, let 
\[ E_Q = \{x \in Q: \Vert A_Q W^{-1/p} (x) \Vert \leq 2 C_1 .\} \]
By Chebychev's inequality and (\ref{equivApcond}),
\[ 2C_1 |Q \setminus E_Q |  \leq \int_{Q \setminus E_Q}   \Vert A_Q W^{-1/p} (x) \Vert \, dx \leq \int_{Q}   \Vert A_Q W^{-1/p} (x) \Vert \, dx \leq C_1 |Q|. \]
Thus, $|Q \setminus E_Q| \leq |Q|/2$, so $|E_Q| \geq |Q|/2$.
By \cite[Proposition 2.7]{FJ2} and the inequality $|A_Q \vec{s}_Q|  \leq  \Vert A_Q W^{-1/p} \Vert | W^{1/p} \vec{s}_Q|$,
\[ \Vert \vec{s} \Vert_{\dot{f}^{\alpha q}_p ( \{ A_Q \} )} \leq c \left\Vert \left(  \sum_{Q \in \mathcal{D} } \left(   |Q|^{-\al /n -1/2} \vert \, A_Q \vec{s}_Q \, \vert \chi_{E_Q} \right)^q  \right)^{1/q} \right\Vert_{L^p} \]
\[ \leq 2c \, C_1 \left\Vert \left(  \sum_{Q \in \mathcal{D} } \left(   |Q|^{-\al /n -1/2} \vert \, W^{1/p} \vec{s}_Q \, \vert \chi_{E_Q} \right)^q  \right)^{1/q} \right\Vert_{L^p} \leq 2c\, C_1 \Vert \vec{s} \Vert_{\dot{f}^{\alpha q}_p ( W )}. \]
\end{proof}

If we assume that $W \in A_p$, the proof of Theorem \ref{firstAQimbed} can be modified slightly to estimate $\Vert  \left\{  \langle \vec{f}, \tilde{\ffi}_{Q} \rangle \right\}_{Q \in \mathcal{Q}} \Vert_{\dot{f}^{\alpha q}_p (\{ A_Q \})}$ by $\Vert \vec{f} \Vert_{\dot{F}^{\alpha q}_p (W )}$.


\begin{Thm}\label{AQ-Wfuncspimbed}
Suppose $0<p<\infty, 0 < q \leq \infty, \alpha \in \R, \ffi \in \mathcal{A}$, and $W \in A_p$. Suppose $\{ A_Q \}_{Q \in \mathcal{Q}}$ is sequence of reducing operators of order $p$ for $W$. Then there exists $c>0$ such that for all $\vec{f}  \in \dot{F}^{\alpha q}_p ( W )$,
\beq \label{AQWest}
\Vert \vec{f} \, \Vert_{\dot{F}^{\alpha q}_p ( \{ A_Q \} )} \leq c \Vert \vec{f} \, \Vert_{\dot{F}^{\alpha q}_p ( W )} .
\eeq 
\end{Thm}

\begin{proof}
Let $\ffi \in \mathcal{A}$ be the test function in the definition of $\dot{F}^{\alpha q}_p ( \{ A_Q \} )$ and $\dot{F}^{\alpha q}_p ( W )$.  Let $\tilde{\ffi} (x)= \overline{\ffi(-x)}$. We will show 
\beq \label{triebestforW}
\left\Vert  \left\{  \langle \vec{f}, \tilde{\ffi}_{Q} \rangle \right\}_{Q \in \mathcal{Q}} \right\Vert_{\dot{f}^{\alpha q}_p (\{ A_Q \})} \leq c \Vert \vec{f} \Vert_{\dot{F}^{\alpha q}_p (W )} . 
\eeq
Then (\ref{AQWest}) follows from this estimate and (\ref{convtriebelest}) with $\varphi$ replaced by $\tilde{\varphi} \in \mathcal{A}$, noting that $W \in A_p$ implies that any sequence of reducing operators $\{ A_Q \}_{Q \in \mathcal{Q}}$ of order $p$ for $W$ is strongly doubling of order $(\beta, p)$ for some $\beta>0$, by Lemma \ref{doublemma}.

In particular, $\{ A_Q\}$ is weakly doubling of some order $r>0$. Therefore, we have the estimate (\ref{estforlater}), obtained in the proof of Theorem \ref{firstAQimbed}, where $A \in (0,1]$ can be taken arbitrarily small and $R$ can be taken arbitrarily large, depending on $A$ if necessary.  

For $0 < p \leq 1$, we obtain 
\[\left|  A_{Q_{j \ell}} \ffi_j \ast \vec{f} (x)   \right|^A   \leq c \left|  W^{1/p}(x) \ffi_j \ast \vec{f} (x) \right|^A   \,\,\, a.e. \,\, \mbox{on} \,\, Q_{j \ell},  \]
by applying (\ref{equivApcondp<1}).  Substituting this estimate on the right side of (\ref{estforlater}),  
\[ \sup_{x \in Q_{j k}} \left|A_{Q_{j k}} \ffi_j \ast \vec{f} (x) \right|^A  \leq c \sum_{\ell \in \Z^n} (1 + |k - \ell|)^{-A(R-r)} 2^{j n} \int_{Q_{j \ell}} \left|W^{1/p} \ffi_j \ast \vec{f}  \, \right|^A \, dx . \]
Proceeding as in the proof of Theorem \ref{firstAQimbed}, letting $h = \left|2^{j \al} W^{1/p} \ffi_j \ast \vec{f} \,  \right|^A$ in (\ref{newmaxineq}), we obtain (\ref{triebestforW}). 

For $1< p < \infty$, we apply H\"{o}lder's inequality with exponents $t= p^{\, \prime}/(p^{\, \prime}-A)$ and $t^{\, \prime} = p^{\, \prime}/A$, where $p^{\, \prime} = p/(p-1)$, to obtain, for $Q = Q_{j \ell}$,
\[  \int_Q \left| A_Q \ffi_j \ast \vec{f}   \, \right|^A \, dx \leq 
 \int_Q \Vert A_Q W^{-1/p} \Vert^A \left|  W^{1/p} \ffi_j\ast \vec{f} \,  \right|^A \, dx  \]
\[ \leq c \left(  \int_Q \Vert A_Q W^{-1/p} \Vert^{p^{\prime}} \, dx \right)^{A/p^{\prime}}   \left( \int_Q \left|  W^{1/p} \ffi_j \ast \vec{f} \, \right|^{Ap^{\prime}/(p^{\prime}-A)}\, dx \right)^{(p^{\prime}-A)/p^{\prime}} \]
\[ \leq c 2^{-jn/t^{\, \prime}} \left(  \int_Q \left| W^{1/p} \ffi_j \ast \vec{f} \,  \right|^{Ap^{\prime}/(p^{\prime}-A)} \, dx \right)^{(p^{\prime}-A)/p^{\prime}}, \]
by (\ref{equivApcond}), which is exactly why we need $W \in A_p$.   From (\ref{estforlater}) we obtain
\[ \sup_{x \in Q_{j k}} \left|A_{Q_{j k}} \ffi_j \ast \vec{f} ( x) \right|^A  \leq c \sum_{\ell \in \Z^n} (1 + |k - \ell|)^{-A(R-r)} \left(  2^{j n} \int_{Q_{j \ell}} \left|W^{1/p} \ffi_j \ast \vec{f}   \right|^{At} \, dx \right)^{1/t} .  \]
Applying H\"{o}lder's inequality with indices $t$ and $t^{\prime}$ gives
\[ \sup_{x \in Q_{j k}} \left|A_{Q_{j k}} \ffi_j \ast \vec{f} ( x) \right|^{At} \leq c \left( \sum_{\ell \in \Z^n} (1 + |k - \ell|)^{-A(R-r)} \right)^{t/t^{\prime}} \]
\[ \times \sum_{\ell \in \Z^n} (1 + |k - \ell|)^{-A(R-r)}   2^{j n} \int_{Q_{j \ell}} \left|W^{1/p} \ffi_j \ast \vec{f}  \, \right|^{At} \, dx   \]
\[ \leq c \sum_{\ell \in \Z^n} (1 + |k - \ell|)^{-A(R-r)}   2^{j n} \int_{Q_{j \ell}} \left|W^{1/p} \ffi_j \ast \vec{f} \,  \right|^{At} \, dx , \]
for $R$ sufficiently large.  Now the proof proceeds as for $0 < p \leq 1$, except with $A$ replaced by $At$.  We note that by decreasing $A$ if necessary, we can guarantee that $t = p^{\prime}/(p^{\prime}-A)$ is sufficiently close to $1$ that we still have $p/(At) >1$ and $q / (At) >1$.  This allows us to use the Fefferman-Stein vector-valued maximal inequality as before to obtain (\ref{triebestforW}).  
\end{proof}

The inequalities converse to those in Theorems \ref{AQ-Wseqspimbed} and \ref{AQ-Wfuncspimbed} require some preparatory lemmas.

\begin{Lemma}\label{carl}
Suppose $\alpha = \{ \alpha_j \}_{j\in \Z}$ is a sequence of non-negative measurable functions on $\R^n$ such that 
\[ \Vert \alpha \Vert_{\mathcal{C}} = \sup_{Q \in \mathcal{D}} \frac{1}{|Q|} \int_Q \sup_{j \in \Z: 2^{- j} \leq \ell(Q) } \alpha_j (x) \, dx < \infty . \]
Then, for any sequence $\{ g_j \}_{j \in \Z}$ of functions on $\R^n$ such that for every $j \in \Z$, $g_j$ is constant on each dyadic cube $Q$ with $\ell(Q) = 2^{-j}$, we have
\begin{equation} \label{alphaineq}
\Vert \sup_{j \in \Z} |\alpha_j g_j | \Vert_{L^1} \leq \Vert \alpha \Vert_{\mathcal{C}} \Vert \sup_{j \in \Z} | g_j | \Vert_{L^1} . 
\end{equation}
\end{Lemma}

\begin{proof}
Let $\beta_j (x) = \sup_{j^{\, \prime} \in \Z: j^{\, \prime} \geq j} \alpha_{j^{\, \prime}} (x)$, for all $j \in \Z$.  Then
\begin{equation}
\sup_{Q \in \mathcal{D}_j } \frac{1}{|Q|} \int_Q \beta_{j} (x) \, dx \leq \Vert \alpha \Vert_{\mathcal{C}} ,  \label{betacarl}
\end{equation}
for every $j \in \Z$.  We will prove 
\begin{equation}
\Vert \sup_{j \in \Z} |\beta_j  g_j | \Vert_{L^1} \leq \Vert \alpha \Vert_{\mathcal{C}} \Vert \sup_{j \in \Z} | g_{j} | \Vert_{L^1} , \label{betaineq}
\end{equation}
which implies (\ref{alphaineq}).  Note that $\beta_{j + 1} \leq \beta_{j}$ for all $j \in \Z$.  

To prove (\ref{betaineq}), we assume $g_j \geq 0$ for all $j \in \Z$.  We also assume that there exists $N >0$ such that $g_j = 0$ for all $j< -N$.  To see this reduction, suppose we have (\ref{betaineq}) under this extra assumption.  Then for general $\{ g_j \}_{j \in \Z}$, let $g_j^{(N)} = g_j$ if $j\geq - N$ and $g_j^{(N)} = 0$ if $j<- N$.  Then $\sup_{j \in \Z}  \beta_j g_j^{(N)} (x)$ is nondecreasing in $N$ and converges to $\sup_{j \in \Z}  \beta_j g_j (x)$ for each $x$ as $N \rightarrow \infty$.  Applying the monotone convergence theorem to both sides of the inequality $\Vert \sup_{j \in \Z} (\beta_j g_j^{(N)})  \Vert_{L^1}  \leq \Vert \alpha \Vert_{\mathcal{C}} \Vert \sup_{j \in \Z}  g_j^{(N)}  \Vert_{L^1}$ yields (\ref{betaineq}).  We also assume that $g_{j + 1} \geq g_{j}$ for all $j \in \Z$.  If the result is known in this case, then for general $\{ g_j \}_{j \in \Z}$, we let $h_j = \sup_{j^{\, \prime} \in \Z: j^{\, \prime} \leq j} g_{j^{\, \prime}}$, so that the sequence $\{ h_j \}_{j \in \Z}$ is nondecreasing, and still satisfies the condition that $h_j$ is constant on dyadic cubes of side length $2^{-j}$.  Then 
\[  \Vert \sup_{j \in \Z} ( \beta_j g_j ) \Vert_{L^1} \leq  \Vert \sup_{j \in \Z} ( \beta_j h_j ) \Vert_{L^1} \leq  \Vert \alpha \Vert_{\mathcal{C}} \Vert \sup_{j \in \Z}  h_j  \Vert_{L^1} = \Vert \alpha \Vert_{\mathcal{C}} \Vert \sup_{j \in \Z}  g_j  \Vert_{L^1} . \]

We will show that 
\begin{equation}
 \Vert \max_{j: -N \leq j \leq \ell} ( \beta_j g_j ) \Vert_{L^1} \leq  \Vert \alpha \Vert_{\mathcal{C}} \Vert g_{\ell}  \Vert_{L^1}  \label{finiteest} 
\end{equation}
by induction on $\ell$ starting with $\ell = -N$. Then letting $\ell \rightarrow \infty$ and applying the monotone convergence theorem completes the proof.  The case $\ell = -N$ is easy; writing $g_{-N} = \sum_{k \in \Z^n} c_{-N, k} \chi_{Q_{-N, k}}$, where each $c_{-N,k}$ is a non-negative constant, we have
\[ \Vert  \beta_{-N} g_{-N}  \Vert_{L^1} = \sum_{k \in \Z^n} c_{-N, k} \int_{Q_{-N, k}} \beta_{-N} \]
\[ \leq \Vert \alpha \Vert_{\mathcal{C}} \sum_{k \in \Z^n} c_{-N, k} | Q_{-N, k} | = \Vert \alpha \Vert_{\mathcal{C}} \Vert   g_{-N}  \Vert_{L^1}, \]
by (\ref{betacarl}).  

Now we assume (\ref{finiteest}) for $\ell$. To prove it for $\ell +1$, note that because the $g_j$'s are nondecreasing and constant on dyadic cubes of side length $2^{-j}$, we can write 
\[ g_{\ell +1}(x) = g_{\ell}(x) + \sum_{k \in \Z^n} d_{\ell +1, k} \chi_{Q_{\ell +1, k}}(x) , \]
where each $d_{\ell +1, k}$ is a non-negative constant.  Hence, 
\[ \beta_{\ell +1} g_{\ell +1} = \beta_{\ell +1} g_{\ell} + \beta_{\ell +1} \sum_{k} d_{\ell +1, k} \chi_{Q_{\ell +1, k}} \leq  \beta_{\ell } g_{\ell} + \beta_{\ell + 1} \sum_{k} d_{\ell +1, k} \chi_{Q_{\ell +1, k}}, \]
because the $\beta_j$ are nonincreasing.  Therefore,
\[ \max_{j: -N \leq j \leq \ell + 1} \beta_j g_j \leq \beta_{\ell +1} \sum_{k} d_{\ell +1, k} \chi_{Q_{\ell +1, k}} + \max_{j: -N \leq j \leq \ell} \beta_j g_j . \]
Consequently, 
\[ \Vert \max_{j: -N \leq j \leq \ell + 1} ( \beta_j g_j ) \Vert_{L^1} = \sum_{m \in \Z^n} \int_{Q_{\ell, m}} \max_{j: -N \leq j \leq \ell + 1} ( \beta_j g_j ) \]
\[  \leq \sum_{m \in \Z^n} \sum_{k: Q_{\ell +1, k} \subseteq Q_{\ell, m}} d_{\ell +1, k} \int_{Q_{\ell +1, k}} \beta_{\ell +1} +  \sum_{m \in \Z^n} \int_{Q_{\ell, m}} \max_{j: -N \leq j \leq \ell } ( \beta_j g_j ) \]
\[ \leq \Vert \alpha \Vert_{\mathcal{C}} \sum_{m \in \Z^n} \sum_{k: Q_{\ell +1, k} \subseteq Q_{\ell, m}} d_{\ell +1, k} |Q_{\ell +1, k}| + \int_{\R^n} \max_{j: -N \leq j \leq \ell } ( \beta_j g_j )  \]
\[\leq \Vert \alpha \Vert_{\mathcal{C}} \sum_{m \in \Z^n} \int_{Q_{\ell, m}} ( g_{\ell +1} - g_{\ell} ) + \Vert \alpha \Vert_{\mathcal{C}} \int_{\R^n} g_{\ell}    = \Vert \alpha \Vert_{\mathcal{C}} \Vert g_{\ell + 1}  \Vert_{L^1}, \]
by (\ref{betacarl}) and the induction hypothesis.  This completes the induction step and hence the proof.
\end{proof}

For $0<p,q \leq \infty$, we define the space $L^p (\ell^q)$ to consist of all sequences $\{ f_j \}_{j \in \Z}$ of scalar-valued measurable functions on $\R^n$ such that
\[ \Vert \{ f_j \} \Vert_{L^p (\ell^q)}  =  \Vert \Vert \{ f_j \} \Vert_{\ell^q(\Z)} \Vert_{L^p(\R^n)} < \infty .  \]
We define $E_j$, the averaging operator at level $j$, acting on a locally integrable function $f$ on $\R^n$, by
\[ E_j (f) = \sum_{Q \in \mathcal{D}_j} \left( \frac{1}{|Q|} \int_Q f  \right) \chi_Q . \]

\begin{Thm}\label{Naz} (Nazarov)  Suppose $\{ \gamma_j \}_{j \in \Z}$ is a sequence of non-negative measurable functions on $\R^n$. 
  
\vspace{0.1in}  
  
(i)  Suppose $0 < q \leq p < \infty$, and $\{ \gamma_j \}_{j \in \Z}$ satisfies 
\begin{equation}
\sup_{Q \in \mathcal{D}_j} \frac{1}{|Q|} \int_{Q} \gamma_j^{p (1+\delta)} \leq c,  \label{firstgamcond}
\end{equation}
for some $c, \delta >0$, independent of $j \in \Z$.  Then there exists $c>0$ such that for any sequence $\{ f_j \}_{j \in \Z}$ of measurable functions on $\R^n$,
\begin{equation}
\Vert \{ \gamma_j E_j (f_j) \} \Vert_{L^p ( \ell^q)} \leq c \Vert \{  E_j (f_j) \} \Vert_{L^p ( \ell^q)}.  \label{firstgammaest}
\end{equation}
If $1 \leq p < \infty$, we also have 
\begin{equation}
\Vert \{ \gamma_j E_j (f_j) \} \Vert_{L^p ( \ell^1)} \leq c \Vert \{  f_j \} \Vert_{L^p ( \ell^1)}.  \label{secgammaest}
\end{equation}

\vspace{0.1in}

(ii) Suppose $1 < p < \infty, 1 \leq q \leq \infty$, and $\{ \gamma_j \}_{j \in \Z}$ satisfies 
\begin{equation}
\sup_{Q \in \mathcal{D}} \frac{1}{|Q|} \int_{Q} \sup_{j \in \Z: 2^{- j} \leq \ell(Q)} \gamma_j^{p(1 + \delta) } < \infty ,  \label{secgamcond}
\end{equation}
for some $\delta >0$.  Then there exists $c>0$ such that for any sequence $\{ f_j \}_{j \in \Z}$ of measurable functions on $\R^n$,
\begin{equation}
\Vert \{ \gamma_j E_j (f_j) \} \Vert_{L^p ( \ell^q)} \leq c \Vert \{  f_j \} \Vert_{L^p ( \ell^q)}.  \label{thirdgammaest}
\end{equation}
\end{Thm}

\begin{proof}
We begin with (\ref{firstgammaest}).  Let $t = p/q \geq 1$, and let $t^{\prime}$ be the conjugate index to $t$.  Observe that $E_j(f_j)$ is constant on any $Q \in \mathcal{D}_j$; we denote that constant value by $(E_j(f_j))_Q$.  Then
\[ \Vert \{ \gamma_j E_j (f_j) \} \Vert_{L^p ( \ell^q)}^q = \left\Vert \sum_{j \in \Z} |\gamma_j E_j (f_j) |^q  \right\Vert_{L^t}   = \sup_{\Vert g \Vert_{L^{t^{\prime}}} \leq 1} \left|\int_{\R^n} \sum_{j \in \Z} |\gamma_j E_j (f_j) |^q g  \right|  \]
\[ =  \sup_{\Vert g \Vert_{L^{t^{\prime}}} \leq 1} \sum_{j \in \Z}  \sum_{Q \in \mathcal{D}_j} | (E_j (f_j) )_Q |^q \int_{Q}    \gamma_j^q  | g  | . \]
We use H\"{o}lder's inequality with exponents $t_1 = p (1+\delta)/q = t(1 + \delta) > t$ and $t_1^{\prime}= t_1/(t_1 -1)$ to obtain, for $Q \in \cD_j$,
\[ \frac{1}{|Q|} \int_Q \gamma_j^q  | g  | \leq \left( \frac{1}{|Q|}  \int_Q \gamma_j^{p(1 + \delta)} \right)^{1/t_1} \left( \frac{1}{|Q|}  \int_Q |g|^{t_1^{\prime}} \right)^{1/t_1^{\prime}}  \leq \frac{c}{|Q|} \int_Q \left( M (|g|^{t_1^{\prime}}) \right)^{1/t_1^{\prime}} , \]
by (\ref{firstgamcond}) and because $\left( \frac{1}{|Q|} \int_Q |g|^{t_1^{\prime} } \right)^{1/t_1^{\, \prime}} \leq \left( M (|g|^{t_1^{\prime}})  \right)^{1/t_1^{\, \prime}} (x)$ for all $x \in Q$.  Substituting above gives
\[ \Vert \{ \gamma_j E_j (f_j) \} \Vert_{L^p ( \ell^q)}^q  \leq c \sup_{\Vert g \Vert_{L^{t^{\prime}}} \leq 1} \sum_{j \in \Z}  \sum_{Q \in \mathcal{D}_j} | (E_j (f_j) )_Q |^q \int_Q \left( M (|g|^{t_1^{\prime}}) \right)^{1/t_1^{\prime}} \]
\[ = c \sup_{\Vert g \Vert_{L^{t^{\prime}}} \leq 1}  \int_{\R^n} \sum_{j \in \Z}  \sum_{Q \in \mathcal{D}_j} | (E_j (f_j) )_Q |^q  \chi_Q \left( M (|g|^{t_1^{\prime}}) \right)^{1/t_1^{\prime}}  \]
\[ = c \sup_{\Vert g \Vert_{L^{t^{\prime}}} \leq 1}  \int_{\R^n} \sum_{j \in \Z}   |E_j (f_j)|^q   \left( M (|g|^{t_1^{\prime}}) \right)^{1/t_1^{\prime}} \]
\[ \leq c \sup_{\Vert g \Vert_{L^{t^{\prime}}} \leq 1} \left\Vert \sum_{j \in \Z}   |E_j (f_j)|^q \right\Vert_{L^t} \left\Vert \left( M (|g|^{t_1^{\prime}}) \right)^{1/t_1^{\prime}} \right\Vert_{L^{t^{\prime}}}. \]
Noting that $t_1^{\prime} < t^{\prime}$, and applying the boundedness of the maximal operator on $L^{t^{\prime}/t_1^{\prime}}$, we have $\left\Vert \left( M (|g|^{t_1^{\prime}}) \right)^{1/t_1^{\prime}} \right\Vert_{L^{t^{\prime}}} \leq c \Vert g \Vert_{L^{t^{\prime}}}$.  Recalling that $t = p/q$, we obtain
\[ \Vert \{ \gamma_j E_j (f_j) \} \Vert_{L^p ( \ell^q)}^q  \leq c \left\Vert \sum_{j \in \Z}   |E_j (f_j)|^q \right\Vert_{L^t}  = c \Vert \{  E_j (f_j) \} \Vert_{L^p ( \ell^q)}^q, \]
as desired.

We now consider (\ref{secgammaest}).  It will follow from (\ref{firstgammaest}) for $q=1$ and the claim:
\begin{equation}
\Vert \{ E_j (f_j) \} \Vert_{L^p ( \ell^1)} \leq c \Vert \{  f_j \} \Vert_{L^p ( \ell^1)}  \label{fourthgammaest}
\end{equation}
for $1 \leq p < \infty$.
To prove (\ref{fourthgammaest}), we can assume, by the monotone convergence theorem, that all but finitely many $f_j$ are identically $0$.  Of course we can assume that each $f_j \in L^p (\R^n)$, hence, $E_j (f_j) \in L^p (\R^n)$.  We use the elementary inequality 
\begin{equation}
 \left( \sum_{j \in \Z} |E_j (f_j)| \right)^p \leq p \sum_{j \in \Z} |E_j (f_j)| \left( \sum_{j^{\, \prime} \leq j } |E_{j^{\, \prime}} (f_{j^{\, \prime}} )| \right)^{p-1} \label{sumbyparts} .  
\end{equation}
To prove (\ref{sumbyparts}), note that for any finitely non-zero sequence of non-negative numbers $\{a_j \}_{j \in \Z}$, 
\[ \left( \sum_{j \in \Z} a_j \right)^p = \sum_{j \in \Z} \left( \left( \sum_{j^{\, \prime} = -\infty}^{j} a_{j^{\, \prime}} \right)^p -   \left( \sum_{j^{\, \prime} = -\infty}^{j-1} a_{j^{\, \prime}} \right)^p  \right), \]
and apply $A^p - B^p \leq p (A-B) A^{p-1}$ for $A\geq B \geq 0$.  

Applying (\ref{sumbyparts}), we have 
\[ \Vert \{ E_j (f_j) \} \Vert_{L^p ( \ell^1)}^p \leq p \int_{\R^n} \sum_{j \in \Z} |E_j (f_j)| \left( \sum_{j^{\, \prime} \leq j } |E_{j^{\, \prime}} (f_{j^{\, \prime}} )| \right)^{p-1}  \]
\[ = p \sum_{j \in \Z} \int_{\R^n} \sum_{ Q \in \mathcal{D}_j } | (E_j (f_j))_Q | \chi_Q \left( \sum_{j^{\, \prime} \leq j } |E_{j^{\, \prime}} (f_{j^{\, \prime}} )| \right)^{p-1}  \]
\[ = p \sum_{j \in \Z} \sum_{ Q \in \mathcal{D}_j } \left( \sum_{j^{\, \prime} \leq j } |(E_{j^{\, \prime}} (f_{j^{\, \prime}} ))_Q| \right)^{p-1}  \int_{Q} | E_j (f_j)_Q | , \]
because each $E_{j^{\, \prime}} (f_{j^{\, \prime}})$ is constant on each $Q \in \mathcal{D}_j$ for $j^{\, \prime} \leq j$.  But $\int_{Q} |(E_j(f_j))_Q | = |Q|| (E_j (f_j))_Q | = \left| \int_Q f_j \right| \leq \int_Q |f_j|$.  Substituting above, we have
\[ \Vert \{ E_j (f_j) \} \Vert_{L^p ( \ell^1)}^p  \leq p \sum_{j \in \Z} \sum_{ Q \in \mathcal{D}_j } \left( \sum_{j^{\, \prime} \leq j } |(E_{j^{\, \prime}} (f_{j^{\, \prime}} ))_Q| \right)^{p-1}  \int_{Q} |f_j|  \]
\[ = p \sum_{j \in \Z} \sum_{ Q \in \mathcal{D}_j }  \int_{Q} |f_j| \left( \sum_{j^{\, \prime} \leq j } |(E_{j^{\, \prime}} (f_{j^{\, \prime}} ))_Q| \right)^{p-1}  \]
\[ = p  \int_{\R^n} \sum_{j \in \Z}  \sum_{ Q \in \mathcal{D}_j } |f_j| \chi_Q  \left( \sum_{j^{\, \prime} \leq j } |E_{j^{\, \prime}} (f_{j^{\, \prime}} )| \right)^{p-1}   = p   \int_{\R^n} \sum_{j \in \Z} |f_j|  \left( \sum_{j^{\, \prime} \leq j } |E_{j^{\, \prime}} (f_{j^{\, \prime}} )| \right)^{p-1}  \]
\[ \leq p \left\Vert \sum_{j \in \Z} |f_j| \right\Vert_{L^p} \left\Vert \left(  \sum_{j \in \Z} |E_j (f_j)| \right)^{p-1} \right\Vert_{L^{p/(p-1)}}  = p  \Vert \{  f_j \} \Vert_{L^p ( \ell^1)} \Vert \{ E_j (f_j) \} \Vert_{L^p ( \ell^1)}^{p-1} .   \]
Our assumptions guarantee that $\Vert \{ E_j ( f_j) \} \Vert_{L^p ( \ell^1)} < \infty$, so by dividing, we obtain (\ref{fourthgammaest}).  

We first prove (\ref{thirdgammaest}) in the case $q = \infty$:
\begin{equation} \label{q=inftyest}
\Vert \sup_{j \in \Z} |\gamma_j E_j (f_j) | \Vert_{L^p} \leq c \Vert \sup_{j \in \Z}  |f_j | \Vert_{L^p} .  
\end{equation}
For $j \in \Z$, let $\alpha_j = \gamma_j^p$.  By (\ref{secgamcond}) and Jensen's inequality, $\Vert \alpha \Vert_{\mathcal{C}} < \infty$.  Hence, by Lemma \ref{carl}, 
\[ \Vert \sup_{j \in \Z} | \gamma_j E_j (f_j) | \Vert_{L^p}^p = \Vert \sup_{j \in \Z} \alpha_j |E_j (f_j)|^p \Vert_{L^1} \]
\[ \leq \Vert \alpha \Vert_{\mathcal{C}} \Vert \sup_{j \in \Z}  |E_j (f_j)|^p \Vert_{L^1} = \Vert \alpha \Vert_{\mathcal{C}} \Vert \sup_{j \in \Z} |  E_j (f_j) | \Vert_{L^p}^p  . \]
For $x \in \R^n$, let $Q$ be the dyadic cube of length $2^{- j}$ containing $x$.  Then $|E_j (f_j) (x)| = \left|  |Q|^{-1} \int_Q f_j \right| \leq M(f_j) (x)$.  Thus, 
\[ \Vert \sup_{j \in \Z} |  E_j (f_j) | \Vert_{L^p} \leq \Vert \sup_{j \in \Z} M  (f_j)  \Vert_{L^p} \leq \Vert M ( \sup_{j \in \Z}   (f_j) ) \Vert_{L^p} \leq c \Vert  ( \sup_{j \in \Z}   (f_j) ) \Vert_{L^p} , \]
since $1 < p < \infty$; i.e., (\ref{q=inftyest}) holds. 

The operator taking $\{ f_j \}_{j \in \Z} $ to $\{ \gamma_j E_j (f_j) \}_{j \in \Z}$ is linear.  By (\ref{q=inftyest}), it is bounded on $L^p ( \ell^{\infty})$.  By (\ref{secgammaest}), which holds because (\ref{firstgamcond}) is weaker than (\ref{secgamcond}), this operator is bounded on $L^p (\ell^1)$.  Hence, by complex interpolation (see e.g. \cite{BL}, Theorem 5.12), it is bounded on $L^p (\ell^q)$ for $1 \leq q \leq \infty$, which gives (\ref{thirdgammaest}).  
\end{proof}

\begin{Cor}\label{imbedd}
Suppose $0<p<\infty, 0<  q \leq \infty, W \in A_p (\R^n)$, and let $\{ A_Q \}_{Q \in \mathcal{D}}$ be a sequence of reducing operators of order $p$ for $W$.  For $j \in \Z$, let 
\[ \gamma_{j}(x) = \sum_{Q \in \mathcal{D}_j} \Vert W^{1/p}(x) A_Q^{-1} \Vert \chi_Q (x). \] 
Then there exists $c>0$ such that for any sequence $\{ f_j \}_{j \in \Z}$ of measurable functions on $\R^n$, 
\begin{equation} \label{fifthgammaest}
\Vert \{ \gamma_j E_j (f_j) \} \Vert_{L^p ( \ell^q)} \leq c \Vert \{  E_j (f_j) \} \Vert_{L^p ( \ell^q)}.  
\end{equation}
\end{Cor}

\begin{proof}
Suppose first that $0< q \leq p < \infty$.  Then (\ref{firstgamcond}) holds for some $\delta >0$ by (\ref{revholder}) and Lemma \ref{GoldbergLemmaple1}. Then (\ref{fifthgammaest}) follows from case (i) of Theorem \ref{Naz}.  If $1 < p < \infty, 1 \leq q \leq \infty$, then (\ref{secgamcond}) holds by (\ref{chgolineq}) and Lemma \ref{GoldbergLemmaple1}.  Replacing $f_j$ by $E_j (f_j)$ in (\ref{thirdgammaest}), and noting that $E_j^2 = E_j$, we obtain (\ref{fifthgammaest}) in this case.  

It remains to prove (\ref{fifthgammaest}) for $0 < p \leq 1$, $p < q \leq \infty$.  Pick $A >0$ sufficiently small that $p/A >1$ (and hence, $q/A >1)$.  Then
\[ \Vert \{ \gamma_j  E_j (f_j) \} \Vert_{L^p ( \ell^q)}^A = \Vert \{ \gamma_j^A | E_j (f_j)|^A \} \Vert_{L^{p/A} ( \ell^{q/A})} . \] 
By Lemma \ref{GoldbergLemma}, the sequence $\{ \gamma_j^A \}_{j \in \Z}$ satisfies 
\[ \sup_{Q \in \mathcal{D}} \frac{1}{|Q|} \int_{Q} \sup_{j \in \Z: 2^{- j} \leq \ell(Q)} \left(\gamma_j^A\right)^{p(1 + \delta)/A } < \infty . \]
Thus, applying (\ref{thirdgammaest}) with $p,q$ replaced by $p/A, q/A >1, \gamma_j$ replaced by $\gamma_j^A$, and $f_j$ replaced by $| E_j (f_j)|^A$, noting that $E_j \left( | E_j (f_j)|^A \right) = | E_j (f_j)|^A$, we obtain 
\[ \Vert \{ \gamma_j^A | E_j (f_j)|^A \} \Vert_{L^{p/A} ( \ell^{q/A})} \leq c \Vert \{| E_j (f_j)|^A \} \Vert_{L^{p/A} ( \ell^{q/A})} = c \Vert \{  E_j (f_j) \} \Vert_{L^p ( \ell^q)}^A ,  \]
as desired.
\end{proof}

\begin{Cor}\label{W-AQinequality}
Suppose $\alpha \in \R, 0<p<\infty, 0<  q \leq \infty, W \in A_p (\R^n)$, and $\{ A_Q \}_{Q \in \mathcal{D}}$ is a sequence of reducing operators of order $p$ for $W$.   
Then for any sequence $\vec{s} = \{ \vec{s}_Q \}_{Q \in \mathcal{D}}$, 
\begin{equation}
\Vert \vec{s} \Vert_{\dot{f}^{\alpha q}_p (W)} \leq c \Vert \vec{s} \Vert_{\dot{f}^{\alpha q}_p (\{A_Q \})}  \label{seqW-AQineq}
\end{equation}
and, for any $\vec{f} \in \mathcal{S}^{\prime}/ \mathcal{P}$, 
\begin{equation}
\Vert \vec{f} \Vert_{\dot{F}^{\alpha q}_p (W)} \leq c \Vert \vec{f} \Vert_{\dot{F}^{\alpha q}_p (\{A_Q \})}  \label{funcW-AQineq} .
\end{equation}
\end{Cor}

\begin{proof}
Let $\{ \gamma_j \}_{j \in \Z}$ be as in Corollary \ref{imbedd}.  For $\vec{s} = \{ \vec{s}_Q \}_{Q \in \mathcal{D}}$, define $f_j = \sum_{Q \in \mathcal{D}_j} |Q|^{-\alpha/n -1/2} |A_Q \vec{s}_Q | \chi_Q$.  Note that $f_j$ is constant on each $Q \in \mathcal{D}_j$, hence, $E_j (f_j) = f_j$.  Also define $g_j = \sum_{Q \in \mathcal{D}_j} |Q|^{-\alpha/n -1/2} |W^{1/p} \vec{s}_Q | \chi_Q$ for $j \in \Z$.  Observe that 
\[ g_j \leq \sum_{Q \in \mathcal{D}_j} |Q|^{-\alpha/n -1/2} \Vert W^{1/p} A_Q^{-1} \Vert \vert A_Q \vec{s}_Q \vert \chi_Q = \gamma_j f_j = \gamma_j E_j (f_j) .\]
Thus, by Corollary \ref{imbedd},
\[ \Vert \vec{s} \Vert_{\dot{f}^{\alpha q}_p (W)} = \Vert \{ g_j \} \Vert_{L^p (\ell^q) } \leq \Vert \{ \gamma_j E_j (f_j) \} \Vert_{L^p (\ell^q) } \]
\[ \leq c \Vert \{  E_j (f_j) \} \Vert_{L^p (\ell^q) } = c \Vert \{  f_j \} \Vert_{L^p (\ell^q) } = c \Vert \vec{s} \Vert_{\dot{f}^{\alpha q}_p (\{A_Q \})} . \]
This proves (\ref{seqW-AQineq}).  

To prove (\ref{funcW-AQineq}), define $h_j (x) = 2^{j \alpha} |W^{1/p}(x) \ffi_j \ast \vec{f} (x) |$ and $k_j = \sum_{Q \in \mathcal{D}_j } |Q|^{-\alpha/n} \left( \sup_{x \in Q} \vert A_Q \ffi_j \ast \vec{f} (x) \vert  \right) \chi_Q$ for $j \in \Z$.  Each $k_j$ is constant on cubes $Q \in \cD_j$. Then
\[ h_j \leq \sum_{Q \in \mathcal{D}_j} |Q|^{-\alpha/n} \Vert W^{1/p} A_Q^{-1} \Vert \, \vert A_Q \ffi_j \ast \vec{f} \vert \chi_Q \leq \gamma_j k_j . \]
Hence, by Corollary \ref{imbedd},  
\[ \Vert \vec{f} \Vert_{\dot{F}^{\alpha q}_p (W)} = \Vert \{ h_j \} \Vert_{L^p (\ell^q) } \leq \Vert \{ \gamma_j k_j \} \Vert_{L^p (\ell^q) }
= \Vert \{ \gamma_j E_j (k_j) \} \Vert_{L^p (\ell^q) } \]
\[ \leq c \Vert \{ E_j (k_j) \} \Vert_{L^p (\ell^q) } =  c \Vert \{ k_j \} \Vert_{L^p (\ell^q) } \leq c  \Vert \vec{f} \Vert_{\dot{F}^{\alpha q}_p (\{A_Q \})} ,  \]
where the last step is by Theorem \ref{firstAQimbed}.
\end{proof}

Theorem \ref{AQ-Wseqspimbed}, Theorem \ref{AQ-Wfuncspimbed}, and Corollary \ref{W-AQinequality} yield Theorem \ref{nonavnormeq}.

We make a few remarks about completeness of the spaces we are considering.  If we assume each $A_Q$ is invertible, then $\dot{f}^{\alpha q}_p (\{A_Q \})$ is complete, as follows.  If $\vec{s}^{\,(n)}= \{ \vec{s}^{\,(n)}_Q \}_{Q \in \cD}$ is a Cauchy sequence in $\dot{f}^{\alpha q}_p (\{A_Q \})$, then $\{ A_Q \vec{s}^{\,(n)}_Q \}$ is Cauchy in $\C^m$, hence so is $\vec{s}^{\,(n)}_Q$. Therefore, $\vec{s}^{\,(n)}_Q$ converges to some $\vec{s}_Q$.  Letting $\vec{s}= \{ \vec{s}_Q \}_{Q \in \cD}$, then Fatou's lemma shows that $\vec{s} \in \dot{f}^{\alpha q}_p (\{A_Q \})$ and $\vec{s}^{\,(n)}$ converges to $\vec{s}$ in $\dot{f}^{\alpha q}_p (\{A_Q \})$.  If $W \in A_p$ and we let $\{A_Q \}_{Q \in \cD}$ be a sequence of reducing operators for $W$, then the first equivalence in Theorem \ref{nonavnormeq} shows that $\dot{f}^{\alpha q}_p (W)$ is complete. It follows then from Theorem \ref{wavechar} that for $W \in A_p$, $\dot{F}^{\alpha q}_p (W)$ is complete. Indeed, a Cauchy sequence $\vec{f}_n$ in $\dot{F}^{\alpha q}_p (W)$ has wavelet coefficients, which are Cauchy, and thus, converge in $\dot{f}^{\alpha q}_p (W)$. If we let $\vec{f}$ have the wavelet coefficients of the limit sequence, then $\vec{f} \in \dot{F}^{\alpha q}_p (W)$ and $\vec{f}_n$ converges to $\vec{f}$ in $\dot{F}^{\alpha q}_p (W)$. Finally, if $\{ A_Q\}$ is a sequence of reducing operators for some $W \in A_p$, then Theorem \ref{mainresult} implies that $\dot{F}^{\alpha q}_p (\{A_Q\})$ is complete.

\section{Equivalence of $\dot{F}^{02}_p(W)$ and $L^p(W)$, $1 < p < \infty$} \label{Lpequiv}

One way to prove the classical unweighted Littlewood-Paley characterization of $L^p(\R^n)$ is to demonstrate the boundedness of appropriate Calder\'{o}n-Zygmund operators whose kernels take values in $\mathcal{B} (H_1, H_2)$, the bounded linear transformations from one Hilbert space to another.  This approach originated in \cite{BCP}, and was explicated in \cite{S1}, Ch. II.5 and IV.1 and \cite{S2}, Ch. 6.3-4.  The boundedness of standard (i.e., scalar valued) Calder\'{o}n-Zygmund operators on $L^p(w)$, where $w$ is a scalar $A_p$ weight, was proved by Hunt, Muckenhoupt, and Wheeden in \cite{HMW} in one dimension, and by Coifman and Fefferman in general in \cite{CF}.  The boundedness of Calder\'{o}n-Zygmund operators on $L^p(W)$, where $W$ is a matrix $A_p$ weight, was proved by Christ and Goldberg in \cite{CG} for $p=2$ and by Goldberg in \cite{G}.  Goldberg's proof is an adaptation to the matrix-weight context of Coifman and Fefferman's argument. Theorem \ref{equivwithlp} will be proved by adapting the proof in \cite{G} to the case of kernels with values in $\mathcal{B} (H_1, H_2)$, thus, going from Calder\'{o}n-Zygmund to Littlewood-Paley theory in the matrix-weight setting just as in \cite{BCP} classically. We begin with some unweighted results that we will require. 

Define
\[ \ell^2_m (\Z) = \left\{ w= \{\vec{w}_j \}_{j \in \Z} : \vec{w}_j \in \C^m \,\, \mbox{for all} \,\, j \in \Z, \,\, \mbox{and} \,\, \Vert w \Vert_{\ell^2_m (\Z)} < \infty\right\}, \] 
where $\Vert w \Vert_{\ell^2_m (\Z)} = \left( \sum_{j} |\vec{w}_{j}|^2 \right)^{1/2}.$  
For each $x \in \R^n \setminus \{0\}$, define $K(x) \in \mathcal{B} (\C^m , \ell^2_m (\Z))$ by 
\[  K(x) \vec{z} = \{ \varphi_j (x) \vec{z} \}_{j \in \Z} , \]
for $\ffi$ satisfying (\ref{admiss1}) and (\ref{admiss2}).  We interpret $K$ as the kernel of a convolution operator $T$, where the integration is carried out on each component: for $\vec{f}:\R^n \rightarrow \C^m$, define
\[  T \vec{f} (x) = \int_{\R^n} K(x-y) \vec{f} (y) \, dy = \left\{ \int_{\R^n} \varphi_j (x-y) \vec{f} (y) \, dy  \right\}_{j \in \Z} = \{ \varphi_j \ast \vec{f} (x) \}_{j\in \Z}.   \]
Then 
\[ |T\vec{f}(x)| = \Vert \{ \varphi_j \ast \vec{f} (x) \}_{j\in \Z} \Vert_{\ell^2 (\Z)} = \left( \sum_{j \in \Z} |\varphi_j \ast f (x)|^2  \right)^{1/2} \]
and 
\[  \Vert T\vec{f} \Vert_{L^p (\R^n)} = \left( \int_{\R^n} |Tf|^p \, dx \right)^{1/p} = \left\Vert \left( \sum_{j \in \Z} |\varphi_j \ast f |^2  \right)^{1/2} \right\Vert_{L^p(\R^n)} .\]
Let $|K(x)|$ denote the operator norm of $K(x)$ in $\mathcal{B} (\C^m , \ell^2_m (\Z))$. Then 
\begin{equation} \label{phisizeest}
|K(x)| = \left(\sum_{j \in \Z} |\varphi_j (x)|^2 \right)^{1/2} \leq \sum_{j \in \Z} |\varphi_j (x)|  \leq  \frac{C_{\varphi}}{|x|^n},
\end{equation}
by letting $k \in \Z$ be such that $2^{-k} < |x| \leq 2^{-k+1}$, breaking the sum on $j$ at $-k$, using the estimate $|\ffi_j (x)|\leq C_{\varphi} 2^{jn}$ for $j\leq -k$, the estimate $|\ffi_j (x)|\leq C_{\varphi} 2^{jn}(2^j|x|)^{-n-1}$ for $j>-k$, and summing the resulting geometric series.  Hence, $K$ satisfies the usual Calder\'{o}n-Zygmund size estimate.  

Using (\ref{admiss1}) and (\ref{admiss2}), Plancherel's theorem easily shows that $T$ is $L^2$-bounded, i.e., $\Vert T \vec{f}  \Vert_{L^2(\R^n)} \leq C_{\ffi} \Vert |\vec{f}| \Vert_{L^2(\R^n)}$.  The next step is to prove that $T$ is weak-type 1-1:
\begin{equation} \label{Tweaktypeunweighted}
|\{ x \in \R^n: \Vert T \vec{f} (x)\Vert  > \alpha \}| \leq \frac{C}{\alpha} \int_{\R^n} |\vec{f} (x) |\, dx  .
\end{equation}
This is done, following the now standard approach, as in \cite{S1}, Chapter II.2-3, by applying the Calder\'{o}n-Zygmund decomposition at height $\alpha$ to the scalar function $|\vec{f}|$, obtaining disjoint cubes $\{ Q_k\}_k$.  Define $\vec{g} = \vec{f}$ on $F= \R^n \setminus \cup_k Q_k$ and let $\vec{g}$ be the average of $\vec{f}$ on each $Q_k$.  Let $\vec{b} = \vec{f} - \vec{g}$.  Then $\vec{g} \in L^2 (\R^n)$, and the appropriate weak-type inequality for $\vec{g}$ follows from the $L^2$-boundedness of $T$ and Chebychev's inequality.  
We use the cancellation on $Q_k$ of each component of $\vec{b}_k$, the restriction of $\vec{b}$ to $Q_k$, to subtract within each integral defining $\varphi_j \ast \vec{b_k}$.  The estimate needed then is that for all $y \in Q_k$,
\begin{equation} \label{Hormanderest}
\int_{\R^n \setminus 3Q_k} |K(x-y)-K(x-y_k)| \, dx \leq C,
\end{equation}
where $y_k$ is the center of $Q_k$.  To prove (\ref{Hormanderest}), for each $j$ we apply the mean-value theorem and a standard geometric estimate to obtain 
\[   |\varphi_j (x-y) - \varphi_j (x-y_k)| \leq C  \frac{2^{j(n+1)}\ell(Q_k)}{(1 + 2^j |x-y_k|)^{n+2} }  . \]
Then we apply the imbedding of $\ell^1$ into $\ell^2$, and break the sum on $j$ at $k$, where $|x - y_k| \approx 2^{-k}$, similarly to the proof of (\ref{phisizeest}). Replacing $1 + 2^j |x-y_k|$ by $1$ for $j<k$ and by $2^j |x-y_k|$ for $j\geq k$ and evaluating the resulting geometric series yields 
\begin{equation} \label{Kest}
|K(x-y)-K(x-y_k)| \leq C \ell(Q_k) |x-y_k|^{-n-1},
\end{equation}
which implies (\ref{Hormanderest}).     

Next we need the weak type 1-1 estimate for the maximal operator.  For $\epsilon >0$, let $\varphi_{j, \epsilon} (x) = \varphi_j (x) \chi_{\{x : |x|>\epsilon\}}$.  For each $x \in \R^n \setminus \{0\}$, define $K_{\epsilon}(x) \in \mathcal{B} (\C^m , \ell^2_m (\Z))$ by $ K_{\epsilon}(x) \vec{z} = \{ \varphi_{j,\epsilon} (x) \vec{z} \}_{j \in \Z}$, for $\ffi$ satisfying (\ref{admiss1}) and (\ref{admiss2}).  Then the corresponding operator is $T_{\epsilon}$, which takes $\vec{f}$ to the vector function 
$ T_{\epsilon} \vec{f} = \{\varphi_{j,\epsilon} \ast \vec{f} \, \}_{j \in \Z}$.  Define $|T_{\epsilon} \vec{f}(x)|$ and $\Vert T_{\epsilon} \vec{f} \Vert_{L^p (\R^n)}$ as for $T$ above.  The maximal operator is 
\[ T_* \vec{f}(x) = \sup_{\epsilon >0} | T_{\epsilon} \ast \vec{f} (x) |. \] 
We will need to know that $T_*$ is weak-type 1-1:
\begin{equation} \label{Tstarweaktype}
|\{x \in \R^n :  T_* \vec{f} (x)  > \alpha  \}|  \leq \frac{C}{\alpha} \Vert \vec{f} \Vert_{L^1 (\R^n)} .  
\end{equation}
For the proof of \eqref{Tstarweaktype}, we follow \cite[pp. 34-35]{S2}.  As in \cite{S2}, (\ref{Tstarweaktype}) follows from the inequality
\begin{equation} \label{pointestforTstar}
T_* \vec{f} (x) \leq C  \left(M (| T \vec{f} |^r) (x)\right)^{1/r} +C M (|\vec{f}|) (x) ,
\end{equation} 
for all $x \in \R^n, r>0$, and $\vec{f} \in L^1_{loc} (\R^n)$, where $M$ is the Hardy-Littlewood maximal operator.  To prove (\ref{pointestforTstar}) at a point $\overline{x}$, let $\vec{f}_1 = \vec{f} \chi_{B(\overline{x}, \epsilon)}$ and $\vec{f}_2 = \vec{f}- \vec{f}_1$.  Note that $T_{\epsilon} \vec{f} (\overline{x}) = T \vec{f}_2 (\overline{x})$.  
We first observe that for $x \in B(\overline{x}, \epsilon/2)$ and $y \not\in B(\overline{x}, \epsilon)$, we have $\sum_{j\in \Z} |\varphi_j (x-y) - \varphi_j (\overline{x} -y)|  \leq \frac{C \epsilon}{|y-\overline{x}|^{n+1}}$, by the same argument as for (\ref{Kest}).  Let $A_k = B(\overline{x}, 2^{k} \epsilon ) \setminus B(\overline{x}, 2^{k-1}\epsilon)$ for $k \geq 1$.  Then 
\begin{equation} \label{diffTest}
 | T \vec{f}_2 (x) - T \vec{f}_2  (\overline{x}) | \leq C \epsilon \sum_{k=1}^{\infty} \int_{A_k}  \frac{|\vec{f}(y)|}{|y-\overline{x}|^{n+1}} \, dy \leq  C M (|\vec{f}|) (\overline{x}) . 
\end{equation}
With this estimate and (\ref{Tweaktypeunweighted}), the rest of the proof of (\ref{pointestforTstar}) is just as in \cite{S2}.  Hence, we have (\ref{Tstarweaktype}).

Now let 
\[ W^{1/p} T \vec{f} =  \{W^{1/p} \varphi_j \ast \vec{f} \, \}_{j \in \Z}; \]
i.e., $W^{1/p}$ acts on each component $\varphi_j \ast \vec{f} $.  Note that $\Vert W^{1/p} T \vec{f} \Vert_{L^p(\ell^2)} = \Vert \vec{f} \Vert_{\dot{F}^{0,2}_p (W)}$. 

\begin{Thm} \label{firstdirection}
Suppose $ 1 < p < \infty, \varphi$ satisfies (\ref{admiss1}) and (\ref{admiss2}), and $W \in A_p$.  If $\vec{f} \in L^p(W)$, then $f \in \dot{F}^{0,2}_p (W)$ with 
\begin{equation}\label{leftineq}
 \Vert \vec{f} \Vert_{\dot{F}^{0,2}_p (W)} \leq C \Vert \vec{f} \Vert_{L^p(W)}, 
\end{equation}
where $C$ depends only on $p, \varphi,$ and $W$. 
\end{Thm}

Since our proof follows \cite{G} line-for-line with only a few changes necessary to deal with the Hilbert-space valued kernel involved, we only describe the modifications needed, referring to \cite{G} as much as possible.  For $\epsilon >0$, define  
\[ W^{1/p} T_{\epsilon} \vec{f} =  \{W^{1/p} \varphi_{j, \epsilon} \ast \vec{f} \, \}_{j \in \Z}. \]
Define the associated maximal operator 
\[ (W^{1/p} T)_* \vec{f} (x) = \sup_{\epsilon >0} | W^{1/p}(x) T_{\epsilon} \vec{f} (x) |.   \]
The essence of the proof of Theorem \ref{firstdirection} is the relative distributional inequality in equation (19) of Proposition 4.1 of \cite{G}; we only require the case $q=p$ of that result.
We apply the covering lemma in \cite{G}, p. 212, to the set $E$ defined for our $T$, reducing (19) in \cite{G} to its local version for each cube $Q$ in the covering; i.e., (20)  in \cite{G}.  We select $\overline{x}$ and $B= B(\overline{x}, 3 \,\, \mbox{diam} (Q))$ as in \cite{G}.  We obtain a point $\overline{y} \in Q$ such that 
\[ \max \left(M_w (W^{1/p} \vec{f})(\overline{y}) , M_w^{\, \prime} (W^{1/p} \vec{f})(\overline{y}) \,\right)  \leq c \alpha \,\, \mbox{and} \,\, \Vert V_B W^{-1/p}(\overline{y}) \Vert < \frac{C}{b},    \]
(which is what is intended on p. 213, line 3 of \cite{G}), where $V_B$ is the reducing operator for $B$ and $M_w$ and $M_{w^{\, \prime}}$ are as in \cite{G}, equations (13) and (14).  We let $\vec{f}_1= \vec{f} \chi_B$ and $\vec{f}_2= \vec{f} \chi_{B^c}$.  The proof of the appropriate distributional inequality for $\vec{f}_1$ depends only on the facts (i): $T_*$ commutes with constant matrices, which is true for our $T_*$ as well, since it is true for each component $\varphi_{j, \epsilon} \ast \vec{f}$ of $T_{\epsilon} \vec{f}$, and (ii):  $T_*$ is weak-type 1-1, which is (\ref{Tstarweaktype}) above in our case.  Therefore, we obtain the estimate (21) in \cite{G}. 

For $\vec{f}_2$, we require the estimate 
\begin{equation} \label{pointTest}
| T_{\epsilon} \vec{f}_2 (x) - T_{\epsilon} \vec{f}_2  (\overline{x}) | \leq C M (|\vec{f}|) (\overline{y})
\end{equation}     
for $x \in Q$ and $\epsilon >0$ (compare to \cite[p. 208]{S2}). To obtain \eqref{pointTest}, we have 
\[ | T_{\epsilon} \vec{f}_2  (x) - T_{\epsilon} \vec{f}_2  (\overline{x}) | \leq \sum_{i=1}^3 \int_{E_i}  \sum_{j\in \Z} |\varphi_{j,\epsilon} (x-y) - \varphi_{j, \epsilon} (\overline{x} -y) | |f(y)| \, dy , \]
where $E_1 = \{ y \in B^c: |x-y| > \epsilon, |\overline{x} -y|>\epsilon \}$, $E_2 = \{ y \in B^c: |x-y| \leq \epsilon, |\overline{x} -y| > \epsilon \}$, and 
$E_3 = \{ y \in B^c: |x-y| > \epsilon, |\overline{x} -y| \leq \epsilon \}$.  On the complement of $\cup_{i=1}^3 E_i$, the integrand is $0$.  The integral over $E_1$ is dominated by $C M (|\vec{f}|) (\overline{y})$, by the same argument that established (\ref{diffTest}).  For $y \in E_2$, we have $\varphi_{j, \epsilon}(x-y)=0$ and $|\overline{x}-y| \approx |x-y| \approx |\overline{y}-y| \approx \epsilon$.  Thus, using (\ref{phisizeest}), the integral over $E_2$ above is bounded by
\[ C \int_{E_2}  \frac{|f(y)|}{|\overline{x} -y |^n}  \, dy \leq \frac{C}{\epsilon^n} \int_{B(\overline{y}, C\epsilon)} |f(y)| \, dy \leq C M f (\overline{y}) . \]
The integral over $E_3$ satisfies the same estimate by symmetry.  Hence, (\ref{pointTest}) holds.  Replacing $\vec{f}$ with $V_B \vec{f}$, commuting $V_B$ and $T_{\epsilon}$, and applying the triangle inequality, we obtain 
\[ \Vert V_B T_{\epsilon} \vec{f}_2 (x) \Vert \leq  \Vert V_B T_{\epsilon} \vec{f}_2 (\overline{x}) \Vert +  C M (|V_B \vec{f}|) (\overline{y}) .  \]
Let $\epsilon^{\, \prime} = \max (\epsilon, 3 \ell (Q))$ and note that $T_{\epsilon} \vec{f}_2 (\overline{x}) = T_{\epsilon^{\, \prime}} \vec{f} (\overline{x})$.  This allows us to conclude estimate (22) in \cite{G}.  The remainder of the proof is the same as in \cite{G}, establishing (19) of \cite{G}.  

In the standard way (see \cite{S2}, \S 3.5), the boundedness of $M_w$ and $M_{w^{\, \prime}}$ (\cite{G}, \S 3) and the relative distributional inequality, applied for $\vec{f} \in C^{\infty}_0 (\R^n)$, which then satisfies $( W^{1/p} T )_*\vec{f}  \in L^p (\R^n)$, lead to the inequality 
\[ \Vert  (W^{1/p} T )_*\vec{f} \, \Vert_{L^p(\R^n)} \leq C \Vert W^{1/p} \vec{f} \, \Vert_{L^p(\R^n)} , \]
for $\vec{f} \in C^{\infty}_0 (\R^n)$. For general Calder\'{o}n-Zygmund operators, one has only $|Tf(x)| \leq T_* f(x) + c |f(x)|$, but because of the explicit nature of our operator and the trivial observation that $\lim_{\epsilon \rightarrow 0^+} \varphi_{j, \epsilon} \ast \vec{f}= \varphi_j \ast \vec{f}$, Fatou's lemma yields the simpler conclusion $|W^{1/p} T \vec{f} (x) | \leq (W^{1/p} T)_* \vec{f} (x)$.  Since
\[  \Vert \vec{f} \Vert_{\dot{F}^{0,2}_p (W)}^p = \int_{\R^n}  \Vert \{  W^{1/p} \varphi_j \ast \vec{f} (x)\}_{j \in \Z}   \Vert^p_{\ell^2 (\Z) } \, dx = \int_{\R^n} |W^{1/p} T \vec{f} (x) |^p \, dx , \]
we obtain (\ref{leftineq}) for $\vec{f} \in C^{\infty}_0 (\R^n)$, and a routine density argument as in \cite{G}, p. 215 yields the result for all $\vec{f} \in L^p(W)$.

Theorem \ref{equivwithlp} is a consequence of Theorem \ref{firstdirection} and

\begin{Thm} \label{seconddirection}
Suppose $ 1 < p < \infty, \varphi \in \mathcal{A}$, and $W \in A_p$.  If $\vec{f} \in \dot{F}^{0,2}_p (W) $, then $\vec{f} \in L^p(W)$ and
\[  \Vert \vec{f} \Vert_{L^p(W)} \leq C \Vert \vec{f} \Vert_{\dot{F}^{0,2}_p (W)} , \]
where $C$ depends only on $p, \varphi,$ and $W$. 
\end{Thm}

Since $\dot{F}^{0,2}_p (W)$ is an equivalence class of tempered distributions modulo polynomials, Theorem \ref{seconddirection} is interpreted as follows: given $\vec{f} \in \dot{F}^{0,2}_p (W)$, there is a unique element of the equivalence class of $\vec{f}$ that belongs to $L^p(W)$. 

The proof of Theorem \ref{seconddirection} uses duality.  It is elementary that the dual of $L^p(W)$ is $L^{p^{\, \prime}} (W^{-p^{\, \prime}/p})$ in the sense that 
for each $\vec{g} \in L^{p^{\, \prime}} (W^{-p^{\, \prime}/p})$, the mapping $T_{\vec{g}}: L^p (W) \rightarrow \C$ defined by 
\[  T_{\vec{g}} \vec{f} = \int_{\R^n} \langle \vec{f} (x),  \vec{g} (x) \rangle \, dx  \]
defines a bounded linear functional on $L^p(W)$ with operator norm equal to $\Vert \vec{g} \Vert_{L^{p^{\, \prime}} (W^{-p^{\, \prime}/p})}$, and every bounded linear functional on $L^p(W)$ is of this form, where $\langle \vec{f}(x), \vec{g}(x) \rangle = \sum_{i=1}^m f_i(x) \overline{g_i (x)}$ is the usual dot product of vectors in $\C^m$.  Applying this result to $W^{-p^{\, \prime}/p}$ shows that the dual of $L^{p^{\, \prime}} (W^{-p^{\, \prime}/p})$ is $L^p(W)$ under the same pairing.  

We will consider $\psi$ satisfying the same conditions as $\varphi$ in (\ref{admiss1}) and (\ref{admiss2}).  For $j \in \Z$, let $\psi_j (x) = 2^{jn} \psi (2^j x)$.
For each $j \in \Z$ and $x \in \R^n$, let $\vec{g}_j (x)$ be a vector of length $m$, and assume that each component of $\vec{g}_j$ is a measurable function on $\R^n$.  Define $G = \{ \vec{g}_j \}_{j \in \Z}$.  Define $L^p_W(\ell^2)$ to be the set of all $G = \{\vec{g}_j \}_{j \in \Z}$ such that 
\[  \Vert G \Vert_{L^p_W (\ell^2)} = \left( \int_{\R^n} \left( \sum_{j \in \Z} | W^{1/p} \vec{g}_j (x)|^2 \right)^{p/2} \, dx \right)^{1/p} < \infty. \]
Let $S$ be the operator taking $G$ to the vector function $S(G)$ on $\R^n$ defined by
\[ S(G) = \sum_{j \in \Z} \psi_j \ast \vec{g}_j .\]

\begin{Lemma} \label{Sbnded}
Suppose $ 1 < p < \infty$, and $W \in A_p$.  If $G = \{\vec{g}_j \}_{j \in \Z} \in L^p_W(\ell^2)$, then $S(G) = \sum_{j \in \Z} \psi_j \ast \vec{g}_j \in L^p(W)$ with 
\[ \left\Vert \sum_{j \in \Z} \psi_j \ast \vec{g}_j \right\Vert_{L^p(W)} \leq C \Vert G \Vert_{L^p_W(\ell^2)}, \]
where $C$ depends only on $p, \varphi,$ and $W$. 
\end{Lemma}

\begin{proof}
Let $\tilde{\psi} (x) = \overline{\psi (-x)}$, and for each $j \in \Z$, let $\tilde{\psi}_j (x) = 2^{jn} \tilde{\psi} (2^j x)$.  Suppose $G = \{ \vec{g}_j \}_{j \in \Z} \in L^p_W (\ell^2)$ and $\vec{h} \in L^{p^{\, \prime}} (W^{-p^{\, \prime}/p})$.  Since $W^{-1/p}$ is self-adjoint, 
\[  \left| \left\langle \psi_j \ast \vec{g}_j (x) , \vec{h} (x) \right\rangle \right| = \left| \left\langle  W^{1/p} (x) \vec{g}_j (x) ,W^{-1/p} (x) \tilde{\psi}_j \ast \vec{h} (x) \right\rangle \right| . \]
Bringing absolute values inside the integral and the sum on $j$, using the previous identity, then the Cauchy-Schwarz inequality first for $\langle, \rangle$, then for the sum on $j$, and finally, H\"{o}lder's inequality with indices $p$ and $p^{\, \prime}$ yields 

\vspace{0.2in}

\[  \left| \int_{\R^n} \left\langle \sum_{j \in \Z} \psi_j \ast \vec{g}_j (x) , \vec{h} (x) \right\rangle \, dx   \right| \leq
  \int_{\R^n}  \sum_{j \in \Z} \left| \left\langle \psi_j \ast \vec{g}_j (x) , \vec{h} (x) \right\rangle \right| \, dx  \]
\[  \leq \left( \int_{\R^n}  \left( \sum_{j \in \Z} \left|   W^{1/p} (x) \vec{g}_j (x)\right|^2 \right)^{p/2} \, dx \right)^{1/p}
\left( \int_{\R^n} \left( \sum_{j \in \Z} \left| W^{-1/p} (x) \tilde{\psi}_j \ast \vec{h} (x)  \right|^2 \right)^{p^{\, \prime}/2}  \, dx \right)^{1/p^{\, \prime}} \]
\[ =  \Vert G \Vert_{L^p_W (\ell^2)}  \Vert \vec{h} \Vert_{\dot{F}^{0,2}_{p^{\, \prime}} (W^{-p^{\, \prime}/p})} , \]
where $\dot{F}^{0,2}_{p^{\, \prime}} (W^{-p^{\, \prime}/p})$ is defined with respect to $\tilde{\psi}$, which satisfies the conditions on $\varphi$ in (\ref{admiss1}) and (\ref{admiss2}).  Note that $W^{-p^{\, \prime}/p} \in A_{p^{\, \prime}}$, since $W \in A_p$.  Hence, by Theorem \ref{firstdirection}, 
\[   \Vert \vec{h} \Vert_{\dot{F}^{0,2}_{p^{\, \prime}} (W^{-p^{\, \prime}/p})} \leq C \Vert \vec{h} \Vert_{L^{p^{\, \prime}} (W^{-p^{\, \prime}/p})}. \]

By duality, then,
\[  \left\Vert \sum_{j \in \Z} \psi_j \ast \vec{g}_j \right\Vert_{L^p(W)}= \sup_{\Vert \vec{h} \Vert_{L^{p^{\, \prime}} (W^{-p^{\, \prime}/p})} \leq 1}  \left|  \int_{\R^n} \langle \sum_{j \in \Z} \psi_j \ast \vec{g}_j(x),  \vec{h} (x) \rangle \, dx \right| \]
\[ \leq C  \Vert G \Vert_{L^p_W(\ell^2)} . \]
\end{proof}

To prove Theorem \ref{seconddirection}, given admissible $\ffi$, we define $\psi$ by $\hat{\psi} = \frac{\overline{\hat{\varphi}}}{\sum_{j \in \Z} |\widehat{\varphi_j}|^2}$.  
Then $\psi$ satisfies the conditions on $\ffi$ in (\ref{admiss1}) and (\ref{admiss2}) (this is where the non-degeneracy condition (\ref{admiss3}) is needed), and we have $\sum_{j \in \Z} \widehat{\psi_j} (\xi) \widehat{\varphi_j}(\xi) =1$ for all $\xi \neq 0$.  Roughly, then, the discrete Calder\'{o}n formula $\vec{f} = \sum_{j \in \Z} \psi_j \ast \varphi_j \ast \vec{f} = S (T ( \vec{f}))$ implies 
\[ \Vert \vec{f} \Vert_{L^p(W)} =  \left\Vert S(T(\vec{f})) \right\Vert_{L^p(W)}\leq C \Vert T(\vec{f}) \Vert_{L^p_W (\ell^2)} = C \Vert \vec{f} \Vert_{\dot{F}^{0, 2}_p (W)}.   \]
We detail the convergence issues involved to justify this conclusion as follows.

\begin{proofof} Theorem \ref{seconddirection}.
Let $\vec{f} \in \dot{F}^{0,2}_p (W)$.  For a positive integer $N$, define
$  \vec{F}_N = \sum_{j =-N}^N \psi_j \ast \varphi_j \ast \vec{f}$.  
Applying Lemma \ref{Sbnded} with $\vec{g}_j = \varphi_j \ast \vec{f}$ for $|j|\leq N$ and $\vec{g}_j= \vec{0}$ for $|j| >N$ gives 
\begin{equation} \label{FNlemma}
  \Vert \vec{F}_N \Vert_{L^p(W)} \leq C \left\Vert  \left( \sum_{j=-N}^N |W^{1/p} \varphi_j \ast \vec{f}|^2     \right)^{1/2}  \right\Vert_{L^p(\R^n)} \leq C \Vert \vec{f} \Vert_{\dot{F}^{0 2}_p (W)} . 
  \end{equation}
Hence, $\vec{F}_N \in L^p (W)$ for each $N$.  Using the fact that the supports of $\hat{\psi}_j$ and $\hat{\psi}_k$ overlap only for $|j-k|\leq 1$, it is not difficult to see that $\vec{F}_N$ converges to $\vec{f}$ in $\dot{F}^{0,2}_p(W)$ norm.  

Using Lemma \ref{Sbnded} again and the dominated convergence theorem, we see that $\vec{F}_N$ is Cauchy in $L^p(W)$.  Therefore, $\vec{F}_N$ converges in $L^p(W)$ to some $\vec{H} \in L^p(W)$.  From the imbedding of $L^p (W)$ into $\dot{F}^{0,2}_p (W)$ (Theorem \ref{firstdirection}), it follows that $\vec{H} \in \dot{F}^{0,2}_p(W)$ and  $\vec{F}_N$ converges in $\vec{H}$ in $\dot{F}^{0,2}_p(W)$.  But we know that $\vec{F}_N$ converges in $\vec{f}$ in $\dot{F}^{0,2}_p(W)$.  Hence, $\vec{H} = \vec{f}$, so $\vec{f} \in  L^p(W)$ and $\vec{F}_N$ converges to $\vec{f}$ in $L^p(W)$.  Now we take the limit as $N \rightarrow \infty$ in (\ref{FNlemma}) to obtain $\Vert \vec{f} \Vert_{L^p(W)} \leq  C \Vert \vec{f} \Vert_{\dot{F}^{0 2}_p (W)}  $.
\end{proofof}

\section{Inhomogeneous spaces} \label{inhomog}

As in the unweighted case, there are useful inhomogeneous versions of the spaces under consideration.  The relation between the homogeneous and inhomogeneous spaces is familiar, as in \cite{FJ2}, Section 12, or \cite{R1}, Section 11.  We choose $\Phi \in \cS (\R^n)$ such that supp $\hat{\Phi} \subseteq \{ \xi \in \R^n: |\xi| \leq 2 \}$ and $|\hat{\Phi}(\xi)| \geq c >0$ for $|\xi| \leq 5/3$.  If $\Phi$ satisfies these two conditions and $\varphi \in \mathcal{A}$, we say $(\Phi, \ffi) \in \cA_+$. If $(\Phi, \ffi) \in \cA_+$ we can find $(\Psi, \psi) \in \cA_+$ such that 
\[  \overline{\hat{\Phi}}(\xi) \hat{\Psi} (\xi)+ \sum_{j=1}^{\infty} \overline{\widehat{\varphi_j}}(\xi) \widehat{\varphi_j} (\xi) =1 \,\,\, \mbox{for all} \,\,\, \xi.  \]
For $Q= Q_{0,k}$, for $k \in \Z^n$, define $\Psi_{Q} (x) = \Phi (x-k)$, and similarly for $\Psi$ (which is consistent with (\ref{defphiQ})).  Then the following inhomogeneous $\varphi$-transform identity holds:
\begin{equation} \label{phitranideninhomo}
  \vec{f} =  \sum_{Q \in \cD_0} \langle \vec{f}, \Phi_{Q} \rangle \Psi_Q + \sum_{j=1}^{\infty} \sum_{Q \in \cD_j} \langle \vec{f}, \ffi_{Q} \rangle \psi_Q,   
\end{equation}
where, as usual, the inner product $\langle \vec{f} , \Phi_Q \rangle$ is defined componentwise.  In this case, we have convergence of (\ref{phitranideninhomo}) in $L^2$ if $\vec{f} \in L^2$, in $\mathcal{S}$ if $\vec{f} \in \cS$, and in $\cS^{\, \prime}$ if $\vec{f} \in \cS^{\, \prime}$ (which means that each component of $\vec{f}$ belongs to $L^2, \cS$, or $\cS^{\, \prime}$, respectively).  We note that we don't have to work modulo polynomials because $\widehat{\Phi}(0) \neq 0$.  

For $\alpha \in \R, 0<p<\infty, 0<q \leq \infty$, and $W$ a matrix weight, let $F^{\alpha q}_p(W)$ be the set of all $\vec{f} \in \cS^{\, \prime} (\R^n)$  such that
\[  \Vert \vec{f} \Vert_{F^{\alpha q}_p(W)} = \Vert \Phi \ast \vec{f} \Vert_{L^p(W)} +   \left\Vert \left(  \sum_{j =1}^{\infty} \left|   2^{j \al}  \, W^{1/p} \ffi_j \ast \vec{f} \,  \right|^q  \right)^{1/q} \right\Vert_{L^p(\R^n)} < \infty . \]
If we adopt the convention that $\phi_j = \varphi_j$ for $j \geq 1$, but $\phi_0 = \Phi$, then the equivalence
\[   \Vert \vec{f} \Vert_{F^{\alpha q}_p(W)} \approx  \left\Vert \left(  \sum_{j =0}^{\infty} \left|   2^{j \al}  \, W^{1/p} \phi_j \ast \vec{f} \, \right|^q  \right)^{1/q} \right\Vert_{L^p(\R^n)} \]
shows that $F^{\alpha q}_p(W)$ is obtained by substituting $\Phi$ for $\varphi_0$ and then truncating the expression in the quasi-norm.  

Let $\mathcal{D}_+ = \{ Q \in \cD: \ell(Q)\leq 1 \}$, and suppose $\{A_Q\}_{Q \in \cD_+}$ is a sequence of non-negative $m\times m$ matrices.  Let $F^{\alpha q}_p(\{A_Q \})$ be the set
of all $\vec{f} \in \cS'(\R^n)$ such that
\[ \Vert \vec{f} \, \Vert_{F^{\alpha q}_p(\{A_Q \})} = \left\Vert \sum_{Q \in \cD_0} |A_Q \Phi \ast \vec{f}| \chi_Q \right\Vert_{L^p(\R^n)} \]
\[  + \left\Vert \left(  \sum_{j =1}^{\infty}   \sum_{Q \in \mathcal{D}_j} \left( 2^{j \al} \vert \, A_Q \, \ffi_j \ast \vec{f} \, \vert \chi_Q \right)^q  \right)^{1/q} \right\Vert_{L^p(\R^n)}   .\]
For a sequence $\vec{s} = \{ \vec{s}_Q \}_{Q \in \cD_+}$, where $\vec{s}_Q \in \C^m$ for each $Q \in \cD_+$, we define the quasi-norms $\Vert \vec{s} 
\Vert_{f^{\alpha q}_p(W)} $, for a matrix weight $W$, and $\Vert \vec{s} \Vert_{f^{\alpha q}_p(\{A_Q\})} $, for a sequence $\{ A_Q \}_{Q \in \cD_+}$ of non-negative definite matrices, by replacing the sum over $Q \in \cD$ in the definitions of the corresponding homogeneous quasi-norms by the sum over $Q \in \cD_+$.  Alternatively, define the map $E$ taking $\vec{s} = \{ \vec{s}_Q \}_{Q \in \cD_+}$ to $E\vec{s} = \{ (E\vec{s})_Q \}_{Q \in \cD}$ by $(E \vec{s})_Q = \vec{s}_Q$ if $\ell(Q) \leq 1$, and $(E\vec{s})_Q = \vec{0}$ if $\ell(Q)>1$.  Then $\Vert \vec{s} \Vert_{f^{\alpha q}_p(W)}= \Vert \vec{Es} \Vert_{\dot{f}^{\alpha q}_p(W)}$ and $\Vert \vec{s} \Vert_{f^{\alpha q}_p(\{A_Q\})} = \Vert \vec{Es} \Vert_{\dot{f}^{\alpha q}_p(\{A_Q\})} $.  Then $f^{\alpha q}_p(W)$ and $f^{\alpha q}_p(\{A_Q\})$ are the set of all sequences $\vec{s}$ with finite quasi-norm, respectively.

We have the following analogues for the inhomogeneous spaces $F^{\alpha q}_p (W)$ of our results above.  

\begin{Thm}\label{mainresultinhomo} Suppose $\alpha \in \R, 0 < p < \infty, 0 < q \leq \infty, (\Phi, \ffi) \in \mathcal{A}_+, W \in A_p(\R^n)$, and $\{ A_Q \}_{Q \in \mathcal{D}}$ is a sequence of reducing operators of order $p$ for $W$.  For $\vec{f} \in \cS^{\, \prime}$, let $\vec{s} = \{ \vec{s}_Q \}_{Q \in \mathcal{D}_+}$, where $\vec{s}_Q = \langle \vec{f}, \Phi_{Q} \rangle$ if $\ell(Q)=1$ and $\vec{s}_Q = \langle \vec{f}, \ffi_{Q} \rangle$ if $\ell(Q)<1$.  Then if any of $\Vert \vec{f} \Vert_{F^{\alpha q}_p (W)}$, $\Vert \vec{f} \Vert_{F^{\alpha q}_p (\{A_Q \})}$, $\Vert  \vec{s} \Vert_{f^{\alpha q}_p ( W )}$, or $\Vert\vec{s} \Vert_{f^{\alpha q}_p ( \{ A_Q \} )}$ is finite, then so are the other three, with  
\[ \Vert \vec{f} \Vert_{F^{\alpha q}_p (W)} \approx \Vert \vec{f} \Vert_{F^{\alpha q}_p (\{A_Q \})} \approx \Vert\vec{s} \Vert_{f^{\alpha q}_p ( \{ A_Q \} )} \approx \Vert  \vec{s} \Vert_{f^{\alpha q}_p ( W )}. \]
Also, $F^{\alpha q}_p (W)$ and $F^{\alpha q}_p (\{A_Q \})$ are independent of the choice of $(\Phi, \varphi) \in \mathcal{A}_+$, in the sense that different choices yield equivalent quasi-norms.
\end{Thm}

For MRA wavelet systems, that is, those obtained from a multi-resolution analysis, such as Meyer's wavelets and Daubechies' $D_N$ wavelets, there exists a scaling function, which we call $\Phi_0$, such that $$\{ \Phi_0 (x-k)\}_{k \in \Z^n} \cup \{ \psi^{(i)}_Q \}_{Q \in \mathcal{D}, \ell(Q)<1, 1 \leq i \leq 2^n -1} $$ is an orthonormal basis for $L^2 (\R^n)$, where $\{ \psi^{(i)} \}_{i=1}^{2^n-1}$ are the wavelet generators.  

\begin{Thm}\label{wavecharinhomo}
Suppose $\alpha \in \R, 0 < p < \infty, 0 < q \leq \infty$, and $W \in A_p(\R^n)$.  Suppose that for some sufficiently large positive numbers $N_0, R$, and $S$ (depending on $p, q, \alpha, n$, and $W$), the generators $\{ \psi^{(i)} \}_{1 \leq i \leq 2^n -1} $ of an MRA wavelet system satisfy $\int_{\R^n} x^{\gamma} \psi^{(i)} (x) \, dx =0$ for all multi-indices $\gamma$ with $|\gamma| \leq N_0$, and $|D^{\gamma} \psi^{(i)} (x)| \leq C (1+|x|)^{-R}$ for all $|\gamma| \leq S$. Also suppose $\Phi_0$ satisfies $|D^{\gamma} \Phi_0 (x)| \leq C (1+|x|)^{-R}$ for all $|\gamma| \leq S$.  Let $\vec{s}^{\,(i)}_Q = \langle \vec{f}, \Phi_Q \rangle$ if $\ell(Q)=1$ and $\vec{s}^{\,(i)}_Q = \langle \vec{f}, \psi^{(i)}_Q \rangle$ if $\ell(Q)<1$, and let $\vec{s}^{\,(i)} = \{ \vec{s}^{\,(i)}_Q \}_{Q \in \cD_+}$. Then 
\[ \Vert \vec{f} \Vert_{F^{\alpha q}_p ( W )} \approx \sum_{i=1}^{2^n -1} \Vert\vec{s}^{\,(i)} \Vert_{f^{\alpha q}_p ( W )}  . \]
\end{Thm}

\begin{Thm}\label{equivwithlpinhomo}
Suppose $1 < p < \infty$ and $W \in A_p(\R^n)$.  Then $F^{0 2}_p (W) = L^p(W)$, with equivalent norms.
\end{Thm}

The proofs of Theorems \ref{mainresultinhomo}, \ref{wavecharinhomo}, and \ref{equivwithlpinhomo} are virtually the same as for the homogeneous spaces, by replacing $\varphi_0$ with $\Phi$ and restricting to $j \geq 0$ and $Q \in \cD_+$.  The only property of $\varphi_0$ we used, that $\Phi$ does not satisfy, is that $\varphi_0$ has vanishing moments of all orders.  However, the vanishing moment property of $\varphi_0$ was only needed when dealing with $Q$ having $\ell(Q)>1$, which we do not consider in the inhomogeneous context.  For example, in Theorem \ref{firstAQimbed}, we only use that $\widehat{\varphi_0}$ has support in $B(0, 2)$, which is satisfied by $\Phi$ also.  In the inhomogeneous context, almost diagonal matrices are indexed by $Q,P \in \cD_+$ only, but otherwise their definition is the same.  Their boundedness on $f^{\alpha q}_p ( \{ A_Q \} )$ follows by applying Theorem \ref{ADmatrixbnded} to $E\vec{s}$, defined above. A family of inhomogeneous smooth molecules is defined as before, but only for $\ell(Q) \leq 1$, and molecules $m_Q$ for $\ell(Q)=1$ are not required to satisfy the vanishing moment condition (M1).  For $\ell(P)=1$, the estimates in (\ref{muleqnuest}) for $\varphi_j \ast m_P$ for $j \geq 1$ (or for $\Phi \ast m_P$, replacing $\varphi_0 \ast m_P$) do not require vanishing moments on $m_P$.  Similarly, the estimate (\ref{mugeqnuest}) for $j=0$ and $\ell(P) <1$, but with $\varphi_0$ replaced by $\Phi$, still hold, because this estimate does not require vanishing moments for $\Phi$.  (In general, uses the vanishing moment condition only for the function associated with the smaller cube.)  With these observations, the proof of the inhomogeneous analogue of Theorem \ref{thirdAQest} goes through, using (\ref{phitranideninhomo}) in place of (\ref{phitraniden}) to obtain the inhomogeneous version of (\ref{convtriebelest}).  Similar modifications prove the analogues of Theorems \ref{avgnormeq} and \ref{avgwavechar}.  We restrict Theorem \ref{AQ-Wseqspimbed} to $E\vec{s}$, as above, to obtain its inhomogeneous counterpart.  The proof of Theorem \ref{AQ-Wfuncspimbed} carries over because it only uses the property of $\Phi$ that $\hat{\Phi}$ is supported in $B(0,2)$.  Corollary \ref{imbedd} holds for the inhomogeneous case simply by letting $f_j$ be $0$ for $j<0$. In this way, Theorems \ref{mainresultinhomo} and \ref{wavecharinhomo} follow.

For Theorem \ref{equivwithlpinhomo}, we define $T$ by replacing $\varphi_0$ by $\Phi$ and restricting to $j \geq0$.  Since $\Phi(x)$ satisfies \eqref{admiss1} and \eqref{admiss2}, we still have the Calder\'{o}n-Zygmund estimate \eqref{phisizeest} for the corresponding kernel $K$. The properties of $\Phi$ and $\varphi$ yield the $L^2$ boundedness of $T$ by Plancherel's theorem.  This $L^2$ boundedness and the pointwise estimates are all that is needed for the rest of the Coifman-Fefferman and Goldberg argument, yielding the inhomogeneous version of Theorem \ref{firstdirection}.  The duality argument for Lemma \ref{Sbnded} holds with $\Z$ replaced by $\{ j \in \Z: j \geq 0\}$. Using \eqref{phitranideninhomo} instead of \eqref{phitraniden} then gives the inhomogeneous converse estimate as in Theorem \ref{seconddirection}, completing the proof of Theorem \ref{equivwithlpinhomo}. 

We clarify the relation between the inhomogeneous and homogeneous spaces, at least for $\alpha >0$ and $1 \leq p < \infty$, in Lemma \ref{comparison} below.  Its proof is based on the following lemma.

\begin{Lemma} \label{convolutionest}
Suppose $1\leq p<\infty, W \in A_p$, and $|\varphi(x)| \leq C (1+|x|)^{-(n+1)}$.  Let $\varphi_j (x) = 2^{jn} \varphi (2^j x)$ for $j \in \Z$.  If $\vec{f} \in L^p(W)$, then $\varphi_j \ast \vec{f} \in L^p (W)$ and 
\[  \Vert \varphi_j \ast \vec{f} \, \Vert_{L^p(W)} \leq C \Vert \vec{f} \, \Vert_{L^p(W)},  \]
for some positive constant $C=C(W, \varphi, p)$.
\end{Lemma}

\begin{proof}  First suppose $p>1$.  Recall the maximal operator $M_w$, introduced by Goldberg, defined by 
\[ M_w \vec{f} (x) = \sup_{B: x\in B} \frac{1}{|B|} \int_B |W^{-1/p} (x) W^{-1/p} (y) \vec{f} (y) | \, dy.   \]
Goldberg \cite[Theorem 3.2]{G} proves that if $1<p<\infty$ and $W$ is an $A_p$ weight, then $M_w$ is bounded on the unweighted, vector-valued space $L^p(\R^n)$ .  

Since matrix multiplication commutes with scalar multiplication,   
\[ \left| W^{1/p} \varphi_j \ast \vec{f} (x) \, \right|   \leq \int_{\R^n} |\varphi_j (x-y) W^{1/p} (x) \vec{f}(y)| \, dy .\]
Let $A_0 (x) = B(x, 2^{-j})$ and, for $k \geq 1$, let $A_k (x) = B(x, 2^{k-j}) \setminus B(x, 2^{k-j-1})$.  Then $|\varphi_j (x-y)| \leq c \, 2^{jn} \, 2^{-k(n+1)}$ on $A_k$, so
\[ \left| W^{1/p} \varphi_j \ast \vec{f} (x) \, \right|  \leq   C \sum_{k=0}^{\infty} \frac{2^{jn}}{2^{k(n+1)}} \int_{A_k (x)}  | W^{1/p} (x) \vec{f}(y)| \, dy \]
\[ \leq  C  \sum_{k=0}^{\infty} 2^{-k} \frac{1}{|B(x, 2^{k-j})|} \int_{B(x, 2^{k-j})}  | W^{1/p} (x) W^{-1/p}(y) W^{1/p}(y) \vec{f}(y)| \, dy \]
\[  \leq C   \sum_{k=0}^{\infty} 2^{-k} M_w ( W^{1/p} \vec{f})(x) =  C  M_w ( W^{1/p} \vec{f})(x) . \]
Hence, 
\[ \Vert \varphi_j \ast \vec{f} \, \Vert_{L^p(W)}  \leq C \Vert  M_w ( W^{1/p} \vec{f}) \Vert_{L^p} \leq C \Vert  W^{1/p} \vec{f} \Vert_{L^p} 
=  C \Vert \vec{f} \, \Vert_{L^p(W)} ,\]
by the boundedness of $M_w$.

Now let $p=1$.  Using $|W(x) \vec{f}(y)| = |W(x) W^{-1}(y) W(y)\vec{f}(y)| \leq \Vert W(x) W^{-1}(y) \Vert |W(y) \vec{f}(y)|$ and Fubini's theorem, we obtain 
\[ \Vert \varphi_j \ast \vec{f} \Vert_{L^1 (W)} \leq  \int_{\R^n} \int_{\R^n} |\varphi_j (x-y)| |W(x) \vec{f}(y) | \, dy \, dx \] 
\[  \leq \int_{\R^n} |W(y) \vec{f}(y) | \int_{\R^n} |\varphi_j (x-y)| \Vert W(x) W^{-1}(y) \Vert \, dx \, dy. \]
Let $[W]_{A_1}$denote the supremum on the left side of \eqref{defAp2} when $p=1$.  Let $\{ Q_i \}_{i=1}^{\infty}$ be an enumeration of the cubes in $\R^n$ with sides parallel to the coordinate axes, centers $\vec{z} = (z_1, z_2, \dots , z_n) \in \Q^n$, and side length $\ell(Q_i) \in \Q_+ = \Q\cap (0, \infty)$.  Then there exists a set $E \subset \R^n$ such that for all $y \in \R^n \setminus E$ and all $i \in \N$, we have $\frac{1}{|Q_i|} \int_{Q_i} \Vert W(x) W^{-1} (y) \Vert \, dx \leq [W]_{A_1}$.  Define $A_k (y)$ for $k \geq 0$ as above for $A_k(x)$.  Then for all $y \in \R^n \setminus E$,
\[ \int_{\R^n} |\varphi_j (x-y)| \Vert W(x) W^{-1}(y) \Vert \, dx \leq c \sum_{k=0}^{\infty} \frac{2^{jn}}{2^{k(n+1)}} \int_{A_k (y)} \Vert W(x) W^{-1}(y) \Vert \, dx . \]
For each $k$, we can find $i \in \N$ such that $B(y, 2^{k-j}) \subseteq Q_i$ and $|Q_i| \leq c 2^{(k-j)n}$, with $c$ independent of $k$ and $j$.  Then 
\[ \int_{A_k (y)} \Vert W(x) W^{-1}(y) \Vert \, dx \leq \int_{Q_i}   \Vert W(x) W^{-1}(y) \Vert \, dx \leq  c [W]_{A_1} 2^{(k-j)n}, \]
since $y \in Q_i \setminus E$.  Substituting above, we get 
\[ \int_{\R^n} |\varphi_j (x-y)| \Vert W(x) W^{-1}(y) \Vert \, dx \leq c \sum_{k=0}^{\infty} \frac{2^{jn}}{2^{k(n+1)}}  2^{(k-j)n} =  c\sum_{k=0}^{\infty} 2^{-k} = c, \]
for all $y \not\in E$.  Hence, $\Vert \varphi_j \ast \vec{f} \Vert_{L^1 (W)} \leq  c \int_{\R^n} |W(y) \vec{f}(y) | \, dy$. 
\end{proof}

\begin{Lemma} \label{comparison}
Suppose $\alpha >0, 1\leq p<\infty, 0<q \leq \infty, W \in A_p(\R^n)$, and $\vec{f} \in \cS (\R^n)$.  Then $\vec{f} \in  F^{\alpha q}_p(W)$ if and only if $\vec{f} \in L^p(W)$ and $\vec{f}$ (or the equivalence class mod $\cP$ of $\vec{f}$) belongs to $\dot{F}^{\alpha q}_p(W)$, and we have
\begin{equation} \label{comparisonnorm}
 \Vert \vec{f} \Vert_{F^{\alpha q}_p(W)} \approx  \Vert \vec{f} \Vert_{L^p(W)} + \Vert \vec{f} \Vert_{\dot{F}^{\alpha q}_p(W)} . 
\end{equation}
\end{Lemma} 

\begin{proof}
Substituting the estimate of Lemma \ref{convolutionest} for the standard inequality $\Vert \varphi \ast f \Vert_{L^p (\R^n)} \leq \Vert \varphi  \Vert_{L^1 (\R^n)}\Vert f  \Vert_{L^p (\R^n)}$, the proof follows exactly as the usual proof, outlined in \cite{FJW}, pp. 42-43, so we omit the details.
\end{proof}

\section{Equivalence with Sobolev spaces}\label{Sobolev}

Many of the basic properties of the spaces $\dot{F}^{\alpha q}_p(W)$ can be demonstrated using the results obtained above. In particular, we show how the Riesz potential acts on $\dot{F}^{\alpha q}_p(W)$ and also an equivalence of the matrix-weighted Tribel-Lizorkin spaces with the matrix-weighted Sobolev spaces. 

For $\beta \in \R$, the Riesz potential of order $\beta$ is defined formally as the Fourier multiplier operator $I_{\beta}$ with multiplier $|\xi|^{-\beta}$:  $(I_{\beta}f)\,\hat{}\,(\xi) = |\xi|^{-\beta} \hat{f}(\xi)$.  If $h \in \cS(\R^n)$ satisfies $D^{\alpha} h (0)=0$ for all multi-indices $\alpha$, then $|x|^{-\beta} h(x) \in \cS$.  Thus, by Fourier transform, $I_{\beta}$ maps $\cS_0$ to $\cS_0$.  Hence, $I_{\beta}$ is defined on $\cS^{\, \prime}/\cP = (\cS_0)^*$ by duality: $\langle I_{\beta} f, g\rangle = \langle  f, I_{\beta} g\rangle$ for $f \in \cS^{\, \prime}/\cP$ and $g \in \cS_0$.  We then define $I_{\beta}$ on vector-valued $\vec{f}\in \cS^{\, \prime}/\cP$ componentwise: $I_{\beta} \vec{f} = (I_{\beta}f_1, \dots, I_{\beta}f_m)^T$.

\begin{Prop} \label{Rieszpot}
Suppose $\alpha, \beta \in \R, 0<p<\infty, 0<q \leq \infty$, and $W \in A_p (\R^n)$.  Then $I_{\beta}$ maps $\dot{F}^{\alpha q}_p(W)$ to $\dot{F}^{\alpha + \beta, q}_p(W)$ continuously.
\end{Prop}

\begin{proof}
Let $\varphi \in \mathcal{A}$ be the test function in the definition of $\dot{F}^{\alpha q}_p(W)$.  Since $|\xi|^{-\beta}$ is smooth and nonvanishing on the support of $\hat{\varphi}$, we have $I_{\beta} \varphi \in \mathcal{A}$.  Note that $\varphi_j \ast (I_{\beta} \vec{f}) = (I_{\beta} \varphi_j) \ast \vec{f}$, by Fourier transform.  Defining the dilates $(I_{\beta} \varphi)_j (x) = 2^{jn} (I_{\beta} \varphi) (2^j x)$ as usual, it follows that $I_{\beta} \varphi_j = 2^{-j \beta} (I_{\beta} \varphi)_j$ for each $j \in \Z$.  Hence,
 \[ \Vert \vec{f} \Vert_{\dot{F}^{\alpha + \beta, q}_p (W)} = \left\Vert \left(  \sum_{j \in \Z} \left(   2^{j (\al + \beta)} \vert \, W^{1/p} \ffi_j \ast I_{\beta} \vec{f} \, \vert \right)^q  \right)^{1/q} \right\Vert_{L^p(\R^n)}  \]
 \[ = \left\Vert \left(  \sum_{j \in \Z} \left(   2^{j \al} \vert \, W^{1/p} (I_{\beta} \ffi)_j \ast \vec{f} \, \vert \right)^q  \right)^{1/q} \right\Vert_{L^p(\R^n)}  \leq c     \Vert \vec{f} \Vert_{\dot{F}^{\alpha , q}_p (W)}, \]
by the fact from Theorem \ref{mainresult} that the spaces $\dot{F}^{\alpha q}_p(W)$ are independent of the choice of test function $\varphi \in \mathcal{A}$.
\end{proof}


Let $\partial_{\ell}$ denote the first order distributional partial derivative in the variable $x_{\ell}$, i.e., $\partial_{\ell} f = \frac{\partial f}{\partial x_{\ell}}$, for $\ell=1, 2, \dots, n$.  Let $\partial_{\ell} \vec{f} = (\partial_{\ell} f_1, \dots, \partial_{\ell} f_m )^T$ for $\vec{f} \in \cS^{\, \prime}/\cP$.   

\begin{Prop} \label{homogred}
Suppose $\alpha \in \R, 0<p<\infty, 0<q \leq \infty, W \in A_p (\R^n)$, and $\vec{f} \in \cS^{\, \prime}/\cP (\R^n)$.  Then $\vec{f} \in \dot{F}^{\alpha q}_p(W)$ if and only if $\partial_{\ell} \vec{f} \in \dot{F}^{\alpha-1, q}_p(W)$ for all $\ell = 1, 2, \dots, n$, and we have
\begin{equation} \label{homogrednorm}
\Vert \vec{f}  \Vert_{\dot{F}^{\alpha q}_p(W)} \approx \sum_{\ell =1}^n \Vert \partial_{\ell} \vec{f}  \Vert_{\dot{F}^{\alpha-1, q}_p(W)}.
\end{equation}
\end{Prop}

\begin{proof}
First suppose $\vec{f} \in \dot{F}^{\alpha q}_p(W)$ and let $\ell \in \{ 1, 2, \dots, n\}$.  For $\varphi \in \mathcal{A}$ and $\psi \in \mathcal{A}$ as in (\ref{phitraniden}), let $\vec{s} = \{  \vec{s}_Q \}_{Q \in \mathcal{D}}$, where $\vec{s}_Q= \langle \vec{f}, \psi_Q \rangle$.  Define a sequence $\vec{t}_{\ell} = \{ \vec{t}_{\ell, Q} \}_{Q \in \cD}$ by 
\[ \vec{t}_{\ell, Q} = \langle \partial_{\ell} \vec{f}, \varphi_Q \rangle  = - \langle \vec{f}, \partial_{\ell} \varphi_Q \rangle . \]
Let $\{A_Q\}_{Q \in \mathcal{D}}$ be a sequence of reducing operators of order $p$ for $W$.  By Theorem \ref{mainresult},
\[ \Vert \partial_{\ell} \vec{f} \Vert_{\dot{F}^{\alpha-1, q}_p(W)} \approx \Vert   \vec{t}_{\ell} \Vert_{\dot{f}^{\alpha-1, q}_p(\{A_Q\})}
=  \left\Vert \left( \sum_{Q \in \mathcal{D}} \left( |Q|^{- \frac{\alpha -1}{n} - \frac{1}{2}} |A_Q \vec{t}_{\ell, Q} | \chi_Q    \right)^q \right)^{1/q}  \right\Vert_{L^p} . \]
Applying (\ref{phitraniden}) to $\partial_{\ell} \varphi_Q$ yields 
\[  \partial_{\ell} \varphi_Q  = \sum_{P \in \mathcal{D}} \langle \partial_{\ell} \varphi_Q , \varphi_P \rangle \psi_P = -  \sum_{P \in \mathcal{D}} \langle  \varphi_Q , \partial_{\ell} \varphi_P \rangle \psi_P \]
\[ = -  \sum_{P \in \mathcal{D}}\ell(P)^{-1}  \langle  \varphi_Q , (\partial_{\ell} \varphi)_P \rangle \psi_P =  - \ell(Q)^{-1} \sum_{P \in \mathcal{D}} b_{QP} \psi_P, \]
for $(\partial_{\ell} \varphi)_P  (x) = |P|^{-1/2} (\partial_{\ell} \varphi) ((x-x_P)/\ell(P))$ (consistent with (\ref{defphiQ})) and $b_{QP} = \frac{\ell(Q)}{\ell(P)} \langle  \varphi_Q , (\partial_{\ell} \varphi)_P \rangle$.  Letting $B = \{b_{QP} \}_{Q,P \in \mathcal{D}}$, and substituting above, we see that 
\[ \vec{t}_{\ell, Q} =  \ell(Q)^{-1} \langle \vec{f},  \sum_{P \in \mathcal{D}} b_{QP} \psi_P \rangle = \ell(Q)^{-1} \sum_{P \in \mathcal{D}} b_{QP} \vec{s}_P = \ell(Q)^{-1} (B \vec{s})_Q .\]
Therefore, we obtain 
\[  \Vert \partial_{\ell} \vec{f} \Vert_{\dot{F}^{\alpha-1, q}_p(W)} \approx \Vert  B\vec{s} \Vert_{\dot{f}^{\alpha, q}_p(\{A_Q\})} . \]
By \eqref{admiss1}, \eqref{admiss2} and Parseval's formula, $\langle  \varphi_Q , (\partial_{\ell} \varphi)_P \rangle =0$ unless $1/2 \leq \ell(Q)/\ell(P) \leq 2$, and in that case,
\[ \langle  \varphi_Q , (\partial_{\ell} \varphi)_P \rangle = \frac{|P|^{1/2}}{|Q|^{1/2}} \int_{\R^n} \hat{\varphi} (\xi) \overline{(\partial_{\ell} \varphi)\,^{\hat{}}\, (\ell(P) \xi/\ell(Q) ) } e^{-i \left(\frac{x_Q - x_P}{\ell(Q)}  \right)\cdot \xi} \, d \xi . \]
Since Schwartz functions have rapidly decaying Fourier transforms, we see that $B$ is almost diagonal, i.e., $B \in {\bf ad}^{\alpha, q}_p (\beta)$, for any possible $\alpha, q, p$, and $\beta$. (Alternatively, one could apply Lemma \ref{convest}.) Since $A_p$ weights are doubling, Lemma \ref{doublemma} and Theorem \ref{ADmatrixbnded} show that $B$ acts boundedly on $\dot{f}^{\alpha, q}_p(\{A_Q\})$.  Thus, we obtain
\begin{equation}
\label{E:1}   
\Vert \partial_{\ell} \vec{f} \Vert_{\dot{F}^{\alpha-1, q}_p(W)} \leq c \Vert  \vec{s} \Vert_{\dot{f}^{\alpha, q}_p(\{A_Q\})} \approx \Vert \vec{f}  \Vert_{\dot{F}^{\alpha q}_p(W)}, 
\end{equation}
where the last equivalence in \eqref{E:1} is by Theorem \ref{mainresult}, since $\psi \in \mathcal{A}$.

Now suppose $\partial_{\ell} \vec{f} \in \dot{F}^{\alpha-1, q}_p(W)$ for all $\ell \in \{ 1, 2, \dots, n\}$.  Applying the first direction, which was just proved, we have $\partial^2_{\ell} \vec{f} \in \dot{F}^{\alpha-2, q}_p(W)$ for all $\ell$.  Then $I_{-2} \vec{f} = c \sum_{\ell=1}^n \partial^2_{\ell} \vec{f} \in \dot{F}^{\alpha-2, q}_p(W)$ and by \eqref{E:1}
\[\Vert I_{-2} \vec{f} \Vert_{\dot{F}^{\alpha-2, q}_p(W)} \leq c \sum_{\ell=1}^n \Vert \partial_{\ell} \vec{f}  \Vert_{\dot{F}^{\alpha-1, q}_p(W)}.\]
Then by Proposition \ref{Rieszpot}, $\vec{f} = I_2 I_{-2} \vec{f} \in \dot{F}^{\alpha, q}_p(W)$ with $\Vert \vec{f} \Vert_{\dot{F}^{\alpha, q}_p(W)} \leq c \sum_{\ell=1}^n \Vert \partial_{\ell} \vec{f}  \Vert_{\dot{F}^{\alpha-1, q}_p(W)}$.
\end{proof}

\begin{Rem}
By iteration, Proposition \ref{homogred} can be generalized to any higher order mixed partial derivative  $D^\beta = \partial_1^{\beta_1} \partial_2^{\beta_2} \cdot \partial_n^{\beta_n}$ with $\sum_{i=1}^n \beta_i = |\beta|$ to obtain, for $k \in \N$,
$$
\|\vec{f} \, \|_{\dot{F}^{\alpha q}_p(W)} \approx \sum_{|\beta|=k}  \|D^\beta \vec{f} \, \|_{\dot{F}^{\alpha - k, q}_p (W)}.
$$
\end{Rem}

The inhomogeneous analogue of Proposition \ref{homogred} is the following.

\begin{Prop} \label{inhomogred}
Suppose $\alpha >1, 1 \leq p <\infty, 0<q \leq \infty, W \in A_p (\R^n)$, and $\vec{f} \in \cS^{\, \prime}(\R^n)$.  Then $\vec{f} \in F^{\alpha q}_p(W)$ if and only if $\vec{f} \in L^p(W)$ and $\partial_{\ell} \vec{f} \in F^{\alpha-1, q}_p(W)$ for all $\ell = 1, 2, \dots, n$, and we have
\begin{equation} \label{inhomogredest}
\Vert \vec{f}  \Vert_{F^{\alpha q}_p(W)} \approx \Vert \vec{f} \Vert_{L^p(W)} + \sum_{\ell =1}^n \Vert \partial_{\ell} \vec{f}  \Vert_{F^{\alpha-1, q}_p(W)}.
\end{equation}
\end{Prop}

\begin{proof}
First suppose $\vec{f} \in F^{\alpha q}_p(W)$.  By Lemma \ref{comparison}, $\vec{f} \in L^p(W)$ and $\vec{f} \in  \dot{F}^{\alpha q}_p(W)$, with \eqref{comparisonnorm}.  Let $1 \leq \ell \leq n$.  By Proposition \ref{homogred}, $\partial_{\ell} \vec{f} \in \dot{F}^{\alpha-1, q}_p(W)$, with $\Vert \partial_{\ell} \vec{f} \Vert_{\dot{F}^{\alpha-1, q}_p(W)} \leq c \Vert \vec{f} \Vert_{\dot{F}^{\alpha, q}_p(W)}\leq c \Vert \vec{f} \Vert_{F^{\alpha, q}_p(W)}$.  Then 
\[ \Vert \Phi \ast \partial_{\ell} \vec{f} \Vert_{L^p(W)} = \Vert (\partial_{\ell} \Phi) \ast  \vec{f} \Vert_{L^p(W)} \leq c  \Vert   \vec{f} \Vert_{L^p(W)} , \]
by Lemma \ref{convolutionest} with $j=0$ and $\varphi$ replaced by $\partial_{\ell} \Phi \in \cS$.  Also,
\[  \left\Vert \left(  \sum_{j =1}^{\infty} \left(   2^{j (\al -1)} \vert \, W^{1/p} \ffi_j \ast \partial_{\ell} \vec{f} \, \vert \right)^q  \right)^{1/q} \right\Vert_{L^p(\R^n)} \leq \Vert \partial_{\ell} \vec{f}  \Vert_{\dot{F}^{\alpha-1, q}_p(W)} . \]
Then by definition, $\partial_{\ell} \vec{f} \in F^{\alpha-1, q}_p(W)$ with $\Vert \partial_{\ell} \vec{f} \Vert_{F^{\alpha-1, q}_p(W)} \leq c \Vert \vec{f} \Vert_{F^{\alpha q}_p(W)} $.

Now suppose $\vec{f} \in L^p(W)$ and $\partial_{\ell} \vec{f} \in F^{\alpha-1, q}_p(W)$ for all $\ell = 1, 2, \dots, n$.  Since $\alpha >1$, Lemma \ref{comparison} gives that $\partial_{\ell} \vec{f} \in \dot{F}^{\alpha-1, q}_p(W)$ for each $\ell$.  By Proposition \ref{homogred}, $\vec{f} \in \dot{F}^{\alpha, q}_p(W)$. Since $\vec{f} \in L^p(W)$, we obtain $\vec{f} \in F^{\alpha q}_p(W)$ by Lemma \ref{comparison} again.  Checking the norm estimates associated with these embeddings gives the other direction of \eqref{inhomogredest}. 
\end{proof}

We obtain Proposition \ref{Sobolevequiv} from Lemma \ref{comparison} and Propositions \ref{homogred} and \ref{inhomogred}.  

\begin{proofof} Proposition \ref{Sobolevequiv}.
First suppose $k=1$.  By Lemma \ref{comparison}, $\vec{f} \in F^{1 2}_p(W)$ if and only if $\vec{f} \in L^p(W)$ and $\vec{f} \in \dot{F}^{1 2}_p (W)$, with \eqref{comparisonnorm}.  By Proposition \ref{homogred}, $\vec{f} \in \dot{F}^{1 2}_p (W)$ if and only if $\partial_{\ell} \vec{f} \in \dot{F}^{0 2}_p (W)$, with \eqref{homogrednorm}.  By Theorem \ref{equivwithlp}, $\dot{F}^{0 2}_p (W) = L^p(W)$ with equivalent norms.  Therefore, $\vec{f} \in F^{1 2}_p(W)$ if and only if $\vec{f} \in L^p(W)$ and $\partial_{\ell} \vec{f} \in L^p(W)$ for $\ell \in \{1, 2, \dots, n\}$, with 
\[ \Vert \vec{f} \Vert_{F^{1 2}_p (W)} \approx \Vert \vec{f} \Vert_{L^p(W)} + \sum_{\ell=1}^n  \Vert \partial_{\ell} \vec{f} \Vert_{L^p(W)}= \Vert \vec{f} \Vert_{L^p_1 (W)} . \]

The case of general $k$ now follows easily by induction.  Assuming the result for some $k\geq 1$, then by Proposition \ref{inhomogred}, $\vec{f} \in F^{k+1, 2}_p (W)$ if and only if $\vec{f} \in L^p(W)$ and $\partial_{\ell} \vec{f} \in F^{k, 2}_p (W)$ for $\ell =1, \dots, n$.  By the inductive assumption, $\partial_{\ell} \vec{f} \in F^{k, 2}_p (W)$ if and only if $D^{\beta} \partial_{\ell} \vec{f} \in  L^p(W)$ for all $\beta$ such that $|\beta| \leq k$, with appropriate equivalence of norms.  This yields the induction step and completes the proof.
\end{proofof}

\bibliographystyle{amsplain}

\bibliography{matrixwttriebel5black}
 
\end{document}